\documentclass[10pt]{amsart}
\input epsf
\usepackage{amsfonts,amsthm,amsmath,amssymb,latexsym,amsbsy,amscd,times}
\usepackage{graphicx,mathrsfs,color,caption,epstopdf}
\usepackage{marginnote}
\usepackage[all]{xy}
\usepackage{epsfig}
\usepackage{soul}
\usepackage{wrapfig}
\usepackage[backref]{hyperref}
\usepackage{tabularx}
\usepackage{lipsum}
\numberwithin{figure}{section}
\numberwithin{equation}{section}
\hypersetup{colorlinks=true,linkcolor=blue,citecolor=red}

\theoremstyle{definition}

 \newtheorem{definition}{Definition}[section]
 \newtheorem{remark}[definition]{Remark}
 \newtheorem{example}[definition]{Example}

\newcommand{\eps}{\varepsilon}

\newcommand{\dist}{\mathrm{dist}}
\newcommand{\m}{\mathrm{mod}}

\newcommand{\diam}{\mathrm{diam}}
\newcommand{\N}{\mathbb{N}}

\newcommand{\RS}{\hat{\mathbb{C}}}
\newcommand{\Ha}{\mathcal{H}}
\newcommand{\R}{\mathbb{R}}
\newcommand{\G}{\Gamma}
\newcommand{\g}{\gamma}
\newcommand{\K}{\mathcal{K}}
\newcommand{\J}{\mathcal{J}}

\theoremstyle{plain}

 \newtheorem{proposition}[definition]{Proposition}
 \newtheorem{theorem}[definition]{Theorem}
 \newtheorem{corollary}[definition]{Corollary}
 \newtheorem{lemma}[definition]{Lemma}
 
 \hypersetup{colorlinks=true,linkcolor=blue,citecolor=blue}

\author{H. Hakobyan and J. Rehmert}
\date{\today}

\address[Hrant Hakobyan]{Department of Mathematics, Kansas State University\\ Manhattan, KS, 66506-2602, USA.}
\email{hakobyan@math.ksu.edu}
\address[Jonathan Rehmert]{Department of Mathematics, Fort Hays State University\\ Hays, KS, 67601-4099, USA.}
\email{j\_rehmert@fhsu.edu}

 \thanks{The first author was partially supported by the Simons Foundation Collaboration Grant, (638572, H.H.).}
\begin{document}




\title{Quasisymmetric Koebe Uniformization of metric surfaces}

\maketitle


\begin{abstract}
We study when a metric surface $X$ can be mapped quasisymmetrically onto a circle domain $D\subset\mathbb{C}$ with uniformly relatively separated boundary components. Bonk \cite{Bonk} proved that if $X\subset \hat{\mathbb{C}}$ and  the boundary components of $X$ are uniformly relatively separated uniform quasicircles then $X$ is quasisymmetric to a circle domain. Merenkov and Wildrick \cite{Merenkov Wildrick} showed that Bonk's condition is not sufficient in the non-planar case. 
%
%
We prove that under some mild assumptions, a metric surface is quasisymmetric to a circle domain with uniformly relatively separated boundary components if and only if it is 2-TLP. The latter is a version of a condition introduced and studied by Bonk \cite{Bonk}. This answers a question of Merenkov and Wildrick in \cite{Merenkov Wildrick} and it is also a natural generalization of Bonk's result to non-planar metric surfaces. 
\end{abstract}

\setcounter{tocdepth}{1}
\tableofcontents


\section{Introduction}

A homeomorphism $f:X \to Y$ between metric spaces is said to be a \textit{ (weakly) quasisymmetric mapping} if there is a constant $H\geq 1$ such that for every triple of distinct points $x,y,z\in X$ satisfying $d_X(x,y)\leq d_X(x,z)$ we have
\begin{align*}
	\frac{d_Y(f(x),f(y))}{d_Y(f(x),f(z))}\leq H.
\end{align*}
Two metric spaces $X$ and $Y$ are said to be quasisymmetrically equivalent, or simply quasisymmetric, if there exists a quasisymmetric map $f:X\to Y$.

One of the central questions of contemporary geometric mapping theory and analysis in metric spaces is the following problem, cf. \cite{BonkICM}:

\textbf{Quasisymmetric Uniformization Problem.} \textit{Suppose $X$ is a metric space homeomorphic to some ``standard" metric space $Y$. When is $X$ quasisymmetrically equivalent to $Y$?}


In \cite{Ahlfors-Beurling} Beurling and Ahlfors showed that a Jordan curve $J\subset\mathbb{C}$ is a quasicircle, that is, $J$ is quasisymmetric to $\mathbb{S}^1$, if and only if it is of \textit{bounded turning}. Recall that $X$ is said to be of bounded turning if there is a constant $C\geq 1$ such that for all distinct points $x,y\in X$ there is a connected subset $\gamma\subset X$ containing $x$ and $y$ such that 
	\begin{align}\label{bounded-turning}
	\diam(\gamma) \leq C d_X(x,y).
	\end{align}

Tukia and V\"ais\"al\"a obtained a complete characterization of arbitrary, i.e. nonplanar, quasicircles in \cite{Tukia Vaisala}. They showed that $X$ is quasisymmetric to $\mathbb{S}^1$ if and only if it is of bounded turning and is \textit{(metric) doubling}, that is, there exists a fixed $N\in\mathbb{N}$ such that every ball $B$ in $X$ of radius $0<R\leq\diam X$ can be covered by $N$ balls of radius $R/2$. Bounded turning and metric doubling are well-known to be quasisymmetrically invariant properties of metric spaces.

In the case of $Y=\mathbb{S}^2$, that is, the unit sphere in $\mathbb{R}^3$ equipped with the chordal metric, the quasisymmetric uniformization problem has been actively studied since the early 2000s. A fundamental result in this direction is due to Bonk and Kleiner \cite{Bonk Kleiner}, who showed that if $X$ is a metric space homeomorphic to $\mathbb{S}^2$ which is Ahlfors $2$-regular then there is a quasisymmetric mapping $f: X\to\mathbb{S}^2$ if and only if $X$ is linearly locally connected. 
%
%
We refer to Section \ref{Section:prelims} for the precise definition of linear local connectivity and only recall here that a metric space $X$ is \textit{Ahlfors $2$-regular} if 
%
Hausdorff $2$-measure of every ball $B=B(x,r)$ in $X$  is comparable to $r^2$.

Further important developments in characterizing $\mathbb{S}^2$ up to quasisymmetric and quasiconformal mappings, as well as new proofs of Bonk and Kleiner's Theorem have been obtained in \cite{Rajala:unif, Lytchak Wenger, Ntalampekos Romney - Duke, Ntalampekos Romney - JEMS, Ntalampekos  - smooth approximation}.  See also the recent survey
\cite{Ntalampekos - survey}.

\subsection{Quasisymmetric Koebe Uniformization} In this paper we study the quasisymmetric uniformization problem in the case when $Y$ is a circle domain in $\mathbb{S}^2$, meaning that every boundary component of $Y$ is either a point or a Euclidean circle. This is often referred to as the \textit{Quasisymmetric Koebe Uniformization}, in analogy with Koebe’s classical theorem which states that every finitely connected domain in the plane is conformal to a circle domain.
We are mostly interested in the case when the circle domain $Y$ has no more than countably many boundary components. We say that a metric space $X$ is a \textit{countably connected metric surface} if it is homeomorphic to a circle domain $Y$ such that $\mathbb{S}^2\setminus Y$ has countably many connected components.

Note that if $X$ is quasisymmetric to a circle domain $Y\subset \mathbb{S}^2$
then it has to satisfy certain natural necessary conditions.
For instance, all nontrivial boundary components of $X$ would have to be \emph{uniform quasicircles}, i.e. they should satisfy (\ref{bounded-turning}) with the same constant $C\geq 1$. 

%
%
%

Even if $X$ is a countably connected planar domain with boundary components which are uniform quasicircles, there does not necessarily exist a quasisymmetric map of $X$ onto a circle domain. Indeed, let $X=\mathbb{R}^2\setminus \cup_{n=1}^{\infty} Q_n$ where $Q_n$ denotes  the square $\left[n,n+1-\eps_{n}\right]\times\left[0,1-\eps_{n}\right]$ for $n\geq 1$, and $0<\eps_n<1/2$ approaches $0$ as $n\to\infty$.  Then it can be easily seen that $X$ is not quasisymmetric to a circle domain, see e.g. \cite{Bonk}. The key obstruction for the existence of a quasisymmetry in this example is the fact that the boundary components are not separated well enough. To make this notion precise we first define \emph{relative distance} between disjoint compact connected subsets $E$ and $F$ of a metric space $X$ as follows:
\begin{align*}
    \Delta(E,F):=\frac{\dist(E,F)}{\min\{\diam E,\diam F\}}.
\end{align*}
Unless otherwise stated, we will assume that neither of the subsets $E$ and $F$  are  singletons.

A family $\{\g_i\}_{i=1}^{\infty}$ of compact connected non-point subsets of $X$ is said to be \emph{uniformly relatively separated} if there is a constant $c>0$ such that for all distinct $i,j\in\mathbb{N}$ we have
\begin{align}\label{ineq:URS}
    \Delta(\g_i,\g_j)>c.
\end{align}

In \cite{Bonk} Bonk proved that assuming uniform relative separation is enough to establish some very strong  uniformization results in the planar setting.
\begin{theorem}[Bonk \cite{Bonk}]\label{thm:Bonk}
A countably connected domain $X\subset\mathbb{S}^2$ is quasisymmetric to a circle domain with uniformly relatively separated boundary circles if and only if the boundary components of $X$ are uniformly relatively separated uniform quasicircles.
\end{theorem}
Theorem \ref{thm:Bonk} is quantitative in the sense that the quasisymmetry constant $H$ of the mapping $f$ from $X$ onto the circle domain depends only on the constants in \eqref{bounded-turning} and \eqref{ineq:URS}. This fact is crucial, since it allows one to prove the result for finitely connected domains and pass to the limit as the number of boundary components tends to infinity. This also allows one to obtain similar uniformization results for other spaces which can be approximated by finitely connected domains, e.g. the classical Sierpi\'nski carpet.

It is natural to ask if Theorem \ref{thm:Bonk} holds if $X$ is not assumed to be a subset of $\mathbb{S}^2$. Merenkov and Wildrick showed that this is not the case \cite[Thm. 1.5]{Merenkov Wildrick} by constructing a countably connected metric surface $X$ satisfying \eqref{bounded-turning} and \eqref{ineq:URS} which cannot be  embedded into $\mathbb{S}^2$ quasisymmetrically.

In this paper, we characterize a large class of metric surfaces which are quasisymmetric to circle domains in $\mathbb{S}^2$ with uniformly relatively separated boundary components. In doing so, we also extend Bonk’s Theorem ~\ref{thm:Bonk} to the non-planar setting. In particular, in the context of Ahlfors $2$-regular spaces we prove the following.

\begin{theorem}[Quasisymmetric Koebe Uniformization]\label{thm:main-intro}
	Let $X$ be an Ahlfors 2-regular countably connected metric surface. Then $X$ is quasisymmetric to a circle domain $Y\subset\mathbb{S}^2$ with uniformly relatively separated boundary circles if and only if
	\begin{enumerate}
	    \item $\overline{X}$ is compact,
            \item $X$ is of bounded turning,
		\item $X$ is 2-transboundary Loewner.
	\end{enumerate}
\end{theorem}


Note that every circle domain $Y\subset\mathbb{S}^2$ clearly satisfies (1) and (2) above. Since these properties are quasisymmetrically invariant, 
$X$ must satisfy them as well. The key novelty of Theorem \ref{thm:main-intro} is condition $(3)$, which is defined in terms of Schramm's transboundary modulus \cite{Schramm} and which we briefly recall next in the context of Ahlfors $2$-regular spaces. We refer to Section \ref{Section:transboundary-modulus} for a more complete and general discussion of transboundary modulus and its properties.


\subsection{Transboundary Loewner Properties}

Suppose that $X$ is an Ahlfors $2$-regular space which is homeomorphic to a domain in $\RS$, $\partial X$ is compact, and the collection of boundary components of $X$, denoted by $\partial_0 X$, is countable. Denote  by $\overline{X}_{\partial_0 X}$ the topological quotient of $\overline{X}$ obtained by collapsing each element of $\partial_0 X$ to a point. 

Given a family $\G$ of curves in $\overline{X}$ (or more generally in $\overline{X}_{\partial_0 X}$), \textit{transboundary modulus} of $\G$ is defined as
\begin{align}
  \m_{\partial_0 X}\G=\inf_{\rho} \left\{\int_X \rho^2 d\mathcal H^2 + \sum_{x\in \partial_0 X}\rho^2(x) \right \},
\end{align}
where the infimum is over all Borel functions $\rho:\overline{X}_{\partial_0 X}\to[0,\infty)$ which are \textit{admissible for $\G$}, that is, for all $\g\in\G$ we have
\begin{align}\label{ineq:admissible}
    \ell_{\rho}(\g):=\int_{\g\cap X} \rho ds + \sum_{x\in\g\cap \partial_0 X} \rho(x) \geq 1.
\end{align}

We say $X$ satisfies $2$-\textit{transboundary Loewner property} (or is $2$-TLP) if there is a decreasing function $\Psi:(0,\infty) \rightarrow (0,\infty)$ such that for every pair $E$ and $F$ of disjoint, non-degenerate continua in $X$ and every choice of two distinct boundary components $k_i$ and $k_j$ of $X$ we have
\begin{align}\label{2TLP-intro}
    \bmod_{\partial_0 X}\G(E,F;\overline{X}\setminus k_i\cup k_j) \geq \Psi(\Delta(E,F)),
\end{align}
where $\G(E,F;\overline{X}\setminus k_i\cup k_j)$ denotes the family of curves in $\overline{X}$ that connect $E$ and $F$ and do not intersect $k_i\cup k_j$. The parameter $2$ in this definition refers to the fact that the curves in $\overline{X}$ that connect $E$ and $F$ are assumed not to intersect two of the boundary components $k_i$ and $k_j$ of $X$. 

\begin{remark}
Sufficient conditions for the existence of a quasisymmetry from a non-planar countably connected metric surface $X$ onto a circle domain $Y\subset\mathbb{S}^2$ were previously obtained in \cite{Merenkov Wildrick}. In particular, it was shown that an Ahlfors $2$-regular metric surface $X$ satisfies this if, in addition to some natural necessary conditions, the diameters of the boundary components of $X$ satisfy a square-summability condition, see \cite[Thm. 1.4]{Merenkov Wildrick}. Since this condition is neither necessary for quasisymmetric uniformization nor is quasisymmetrically  invariant, they wondered if one could provide quasisymmetrically invariant conditions characterizing Ahlfors $2$-regular countably connected metric surfaces which can be uniformized by circle domains. Condition (3) in Theorem \ref{thm:main-intro} provides such a condition.
\end{remark}

\begin{remark}
In \cite{Hakobyan:Li}, the authors introduced a weaker condition than \eqref{2TLP-intro}, requiring only lower bounds on $\bmod_{\partial_0 X}\G(E,F;\overline{X})$. They referred to this as the \textit{transboundary Loewner property (TLP)} and showed that it is a necessary condition for the existence of a quasisymmetric embedding of a finitely connected metric surface into $\mathbb{S}^2$. Moreover, they conjectured in \cite{Hakobyan:Li} that TLP is also sufficient for $X$ to be quasisymmetrically equivalent to a circle domain in certain cases.
Theorem \ref{thm:main-intro} shows that this is indeed the case if one assumes a slightly stronger 2-TLP condition (\ref{2TLP-intro}).
\end{remark}

This paper is organized as follows. Sections  \ref{Section:prelims} and \ref{Section:transboundary-modulus} discuss the established theory of quasisymmetric mappings and quasiconformal mappings, as well as modulus and the Loewner property. In Sections \ref{Section:tlp}, \ref{Section:TLP-examples} and \ref{Section:N-tlp-examples} we define the transboundary Loewner property and study some examples. In Section \ref{Section:circle-domains} we consider the modulus estimates in circle domains and Section  \ref{Section:proof} contains the proof of the main theorem.

\section{Preliminaries}\label{Section:prelims}
Everywhere below we say a set is countable if it is either countably infinite, finite, or empty. Moreover, unless otherwise stated, everywhere below we will assume that $(X,d,\mu)$ is a metric measure space such that $(X,d)$ is a separable metric space and $\mu$ is a locally finite Borel regular measure.

\subsection{Modulus} 
An \emph{interval} is a non-empty, connected subset of $\R$ \index{interval} (note this includes singletons). Given a metric space $(X,d)$, a curve \index{curve} in $X$ is a continuous function $\gamma: I \rightarrow X$ where $I$ is an interval. We call the curve open, closed, or compact if $I$ is open, closed, or compact respectively. If $I' \subset I$ is an interval, we call the restriction of $\gamma$ to $I'$ a subcurve \index{subcurve} of $\gamma$. We will often abuse notation and denote $\gamma(I)$ by $\gamma$; moreover, we will abuse language and refer to $\gamma(I)$ as a curve. If $\gamma$ is a singleton, we say the curve is constant.

Given a curve $\gamma$ in $X$ we will denote by $\ell(\gamma)$ its length. We will say a curve $\gamma$ is rectifiable if $\ell(\gamma)<\infty$. We say $\gamma$ is locally rectifiable if every compact subcurve of $\gamma$ is rectifiable.

We recall the definition of modulus of families of curves in quite a general context and refer to \cite{Heinonen} and \cite{HKST} for further discussion of the concept and its properties.

Let $\Gamma$ be a collection of curves in $X$, and $\rho:X \rightarrow [0, \infty]$ a Borel function. We say \textit{${\rho}$ is admissible for ${\Gamma}$}, denoted $\rho \wedge \Gamma$, if for all locally rectifiable curves $\gamma \in \Gamma$, 
\begin{align}
 \int_{\gamma} \rho \ ds \geq 1,
\end{align}
where integration is with respect to arclength. The \textit{modulus of ${\Gamma}$} \index{modulus} is defined as
\begin{align}
\bmod(\Gamma) := \inf_{\rho \wedge \Gamma} \int_X \rho^2 d\mu.
\end{align}


It will be convenient to define the following short hand notation when the ambient measure space is clear: $\ell_\rho(\g) = \int_\g \rho \ ds$ and $A(\rho) := \int_X \rho^2 d\mu.$ These quantities will be called the \emph{$\rho$-length} of $\gamma$ and \emph{$\rho$-area} of $X$, respectively.

For $\Gamma_1,\Gamma_2$ families of curves in $X$, we say $\Gamma_1$ \emph{minorizes} \index{minorize} $\Gamma_2$, denoted $\Gamma_1 < \Gamma_2$, if every curve in $\Gamma_2$ contains a subcurve in $\Gamma_1$.

The following properties of modulus are well-known (see \cite{HKST}) and will be used repeatedly below.
\begin{proposition}
\label{modulus properties}
Let $(X,d,\mu)$ be a metric measure space with $(X,d)$ separable and $\mu$ a locally finite Borel regular measure. For each $k \in \N$, let $\Gamma_k$ be a collection of curves in $X$.
\begin{itemize}
    \item[(i)] $\bmod(\emptyset) = 0.$ 
    \item[(ii)] If $\Gamma_1 \subset \Gamma_2$, then $\bmod(\Gamma_1) \leq \bmod(\Gamma_2).$ \hfill{(Monotonicity)} \index{monotonicity of modulus}
    \item[(iii)] If $\Gamma_1 < \Gamma_2$, then $\bmod(\Gamma_1) \geq \bmod(\Gamma_2).$ \hfill{(Overflowing)} \index{overflowing property of modulus}
    \item[(iv)] $\bmod(\cup_{k \in \N} \Gamma_k) \leq \sum_{k \in \N} \bmod(\Gamma_k)$. \hfill{(Subadditivity)} \index{subadditivity of modulus}
    \item[(v)] If there are pairwise disjoint Borel sets, $B_k \subset X$, such that $\gamma_k \subset B_k$ for all $\gamma_k \in \Gamma_k$, then $$\bmod(\cup_{k \in \N} \Gamma_k) = \sum_{k \in \N} \bmod(\Gamma_k).$$ 
\end{itemize}
\end{proposition}

\subsection{Quasiconformality}

Let $(X,d_X,\mu_X),(Y,d_Y,\mu_Y)$ be metric measure spaces with $(X,d_X),(Y,d_Y)$ being separable metric spaces, and $\mu_X,\mu_Y$ locally finite Borel regular measures. For $K \geq 1$, we say a homeomorphism $f:X \rightarrow Y$ is \textit{(geometrically)} \textit{K-quasiconformal} \index{geometric quasiconformality} \index{quasiconformal} if for all families of curves $\G$ in $X$, we have 
\begin{align}\label{geometric qc}
    \frac{1}{K} \bmod(\G) \leq \bmod(f(\G)) \leq K \bmod(\G).
\end{align}

There is also another the so-called \textit{analytic definition of quasiconformality} that will be used below. It will appeal to the theory of Newton-Sobolev functions on metric measure spaces. We will not define the Newton-Sobolev class here, as we are not going to use it directly.

\begin{definition}[\cite{Williams}]
Let $(X,d_X,\mu_X)$ be a metric measure space with $(X,d_X)$ separable and $\mu_X$ a locally finite Borel regular measure. Given a metric measure space 
$(Y,d_Y,\mu_Y)$ and a homeomorphism $f:X \rightarrow Y$, we say a Borel function  $g:X \rightarrow \R$ is an \textit{upper gradient} of $f$ if 
\begin{align}\label{def:upper-gradient}
    \int_\g g \ ds \geq d_Y(f(\g(a)),f(\g(b))),
\end{align}
for all rectifiable curves $\g$ in $X$.
 %
%
 We say $g$ is a \textit{weak upper gradient} \index{weak upper gradient} of $f$ if (\ref{def:upper-gradient}) holds for all rectifiable curves $\g$ in $X$ \emph{except for a family of zero modulus}.

 
 A function $g$ is said to be a \emph{minimal weak upper gradient} if for every weak upper gradient $h$ we have 
 \begin{align*}
    h \geq g
 \end{align*} 
 $\mu_X$-almost everywhere.

Let $K \geq 1$ be given. A homeomorphism $f: X\to Y$  is \textit{(analytically) K-quasiconformal} \index{analytic quasiconformality} \index{quasiconformal} if $f \in N_{\text{loc}}^{1,2}(X:Y)$ (the Newton-Sobolev class, see \cite{HKST:Paper}), and $$g_f(x)^2 \leq K J_f(x)$$ for $\mu_X$-almost every $x \in X$,
where $g_f$ is the minimal weak upper gradient of $f$, and $$J_f(x) := \limsup_{r \rightarrow 0} \frac{\mu_Y(f(B(x,r)))}{\mu_X(B(x,r))}.$$ 
\end{definition}

The following theorem shows that the analytic and geometric definitions of quasiconformality are equivalent in remarkable generality.

\begin{theorem}[Williams \cite{Williams}]
\label{Williams}
Let $(X,d_X,\mu_X),(Y,d_Y,\mu_Y)$ be metric measure spaces with $(X,d_X),(Y,d_Y)$ separable and $\mu_X,\mu_Y$ locally finite Borel regular measures. Let $f:X \rightarrow Y$ be a homeomorphism. Then the following are equivalent with the same constant $K \geq 1$.
\begin{itemize}
    \item $f \in N_{\text{loc}}^{1,2}(X:Y)$, and for $\mu_X$-almost every $x \in X$, $$g_f(x)^2 \leq K J_f(x).$$
    \item For every collection of curves, $\G$, in $X$, we have $$\bmod(\G) \leq K \bmod(f(\G)).$$
\end{itemize}
\end{theorem}


\section{Transboundary Modulus}\label{Section:transboundary-modulus}

Transboundary modulus was first introduced as transboundary extremal length by Schramm \cite{Schramm} and was used for giving an alternative proof to the countable case of Koebe's conjecture. He defined it for domains in the sphere; the definition given here will be more general. 

Let $(X,d)$ be a metric space and $\mu$ a locally finite Borel measure on $X$. 
Let $\mathcal{K} = \{ K_i \}_{i \in I}$ be a finite or countable collection of pairwise disjoint continua  in $X$. We denote $K = \cup_{i \in I} K_i$ and $D = X \setminus K$. Unless otherwise stated, below we will assume that $D$ is homeomorphic to a domain in $\RS$. 

Given $\mathcal{K}$ as above we define an equivalence relation on $X$ as follows:
\begin{align*}
    a\sim b \, \Leftrightarrow \, a=b,\mbox{ or } a,b\in K_i  \mbox{ for some } i\in I. 
\end{align*} 
We will denote by $X_\K$  the quotient space $X/\sim$ equipped with the quotient topology and call $X_\K$ \emph{the $\K$-quotient of $X$}. Let $\pi_\K:X \rightarrow X_\K$ be the quotient map. Notice that $\pi_\K$ is injective on $D$; thus, $\pi_\K|_D$ is a homeomorphism. Let $D_\K := \pi_\K(D)$. For each $i \in I$, $\pi_\K(K_i)$ is a singleton in $X_\K$; define $ k_i  = \pi_\K(K_i)$. Let $k=\cup_{i \in I} k_i$. Suppose we equip $(X,d)$ with a Borel measure $\mu$. We can define a Borel measure $\mu_\K$ on $X_\K$ by letting
\begin{align*}
    \mu_\K(E) 
    &= \mu(\pi_\K^{-1}(D_\K \cap E)) + \sum_{i \in I} \delta_{k_i}(E) \\ 
    &= \mu(\pi_\K^{-1}(E \setminus k)) + \sum_{i \in I} \delta_{k_i}(E \cap k)
\end{align*}
for every measurable set $E \subset X_\K$. Note that we only need $\mu$ to be defined on $D$ for $\mu_\K$ to be well-defined. 

As before, by a curve \index{curve (transboundary)} in $X_\K$, we mean a continuous function $\g:J \rightarrow X_\K$ where $J \subset \R$ is an interval. We may also call $\g$ a transboundary curve\index{transboundary curve}. We will often abuse notation and write $\g(J)=\g$. {If $\g \subset D_\K$, then there is a curve $\g' \subset D$ with $\g = \pi_\K(\g')$}. Sometimes we will abuse notation and write $\g=\g'$ if it is clear that the curve is in $D_\K$. Since $D_\K$ is homeomorphic to a domain, it must be open in $X_\K$. Hence $\g^{-1}(D_\K)$ must be a relatively open subset of $J$; thus, it is a union of relatively open subintervals. Let $\{J_m\}_{m \in I_\g}$ be the collection of relatively open subintervals of $J$ such that 
\begin{align*}
    \bigcup_{m \in I_\g} J_m = \g^{-1}(D_\K).
\end{align*} 
Observe that $I_\g$ is empty if and only if  $\g(t)=k_i$ for some $i \in I$, that is $\g$ is a constant curve in $k$. If $I_\g$ has one element, then $\g \subset D_\K$. For each $m \in I_\g$, define the subcurve $\g_m:J_m \rightarrow D$ of $\g$ so that $\g_m(t)=\pi_\K^{-1}(\g(t)).$ Note that 
\begin{align*}
 \bigcup_{m \in I_\g}\pi_\K(\g_m) = \g \setminus k.
\end{align*}

If $\g$ is non-constant, then $\g_m$ is non-constant for all $m \in I_\g$. We say $\g$ is locally rectifiable relative $\K$ \index{local rectifiability (transboundary)} if $\g_m$ is locally rectifiable for all $m \in I_\g$. Given a Borel function $\rho:D \rightarrow [0,\infty]$, we will define the \emph{$\rho$-length of $\g$ relative $\K$} \index{$\rho$-length (transboundary)} by 
\begin{align*}
    \ell_\rho^\K(\g) := \sum_{m \in I_\g} \ell_\rho(\g_m).
\end{align*} 
For a curve $\g$ in $X$, we will often use the following notational shortcuts for $\{\g_m\}$ corresponding to $\pi_\K(\g)$.
\begin{align*}
    \ell(\g \setminus K) &= \sum_{m \in I_{\pi_\K(\g)}} \ell(\g_m) \\
     \ell_\rho(\g \setminus K) &= \sum_{m \in I_{\pi_\K(\g)}} \ell_\rho(\g_m) \\
     \int_{\g \setminus K} \rho \ ds &= \sum_{m \in I_{\pi_\K(\g)}} \int_{\g_m} \rho \ ds
\end{align*}


Let $\G$ be a family of curves in $X_\K$. A \textit{(transboundary) mass distribution} is a function 
$\varrho:D \cup \{K_i\}_{i\in I}\to[0,\infty]$
such that  $\rho:=\varrho|_D$ is a non-negative Borel function, and $\rho_i \geq 0$ for each $i \in I$. We will call $\rho_i$ the weight corresponding to $K_i$.

Alternatively, a mass distribution can be thought of as a function $\varrho:X_{\mathcal{K}}\to[0,\infty]$ such that  $\rho=\varrho|_{X_{\mathcal{K}} \setminus k}$ is a non-negative Borel function, and $\rho_i:=\varrho(k_i) \geq 0, \forall i \in I$. 

We say $\varrho$ is \textit{admissible for ${\G}$ relative $\mathbf{\K}$} \index{admissibility (transboundary)}, denoted $\varrho \wedge_\K \G$, if for all $\g \in \G$ which are locally rectifiable relative $\K$, we have 
\begin{align}\label{def:tr-admiss}
    \ell_{\varrho}^\K(\g) := \ell_\rho^\K(\g) + \sum_{k_i \in \g} \rho_i \geq 1. 
\end{align}

Given a mass distribution $\varrho$ we define its \emph{mass} by
\begin{align}
\begin{split}
    A_\K(\varrho) &:= \int_{X \setminus K} \rho^2 d\mu + \sum_{i \in I} \rho_i^2\\& \,=\int_{X_\K} \varrho^2  d\mu_\K.
\end{split}
\end{align} 

 \textit{Transboundary modulus of ${\G}$} is defined as  
\begin{align}\label{def:trmod}
\begin{split}
        \bmod_\K(\G)&:= \inf_{\varrho \wedge_\K \G} A_\mathcal{K}(\varrho) \\ &\,= \inf_{\varrho \wedge_\K \G} \bigg(\int_{X \setminus K} \rho^2 d\mu + \sum_{i \in I} \rho_i^2 \bigg).
\end{split}
\end{align} 
If $\G$ is a family of curves in $X$, we write $\varrho \wedge_\K \G$ if $\varrho \wedge_\K (\pi_\K(\G))$, and we define $$\bmod_\K(\G) := \bmod_\K(\pi_\K(\G)).$$

In general, transboundary modulus takes on values in $[0,\infty]$. Observe that if $\K$ is empty, or if every $\g \in \G$ is disjoint from $K$, then this definition coincides with the definition of modulus. Much like modulus, if there are no admissible mass distributions for $\Gamma$, then $\bmod_\K(\Gamma) = \infty$. This happens, for example, if $\Gamma$ contains a constant curve outside of $K$. Note that if $\Gamma$ has no curves which are locally rectifiable relative $\K$, then the zero distribution is admissible for $\Gamma$ and $\bmod_\K(\Gamma)=0$. 

Next we list some of the useful properties of transboundary modulus analogous to those of the classical modulus. We do not present the proofs of these properties since they are well-known and follow the similar arguments for classical modulus, cf. \cite{Bonk,Hakobyan:Li,Merenkov}.

\begin{proposition}
\label{transboundary properties}
Let $(X,d,\mu)$ be a metric measure space with $\mu$ a locally finite Borel regular measure. Fix a countable collection of pairwise disjoint continua $\K=\{K_i\}_{i \in I}$, let $K=\cup_{i \in I}K_i$, and suppose $X \setminus K$ is homeomorphic to a domain in $\RS$. For each $n \in \N$, let $\Gamma_n$ be a collection of curves in $X$.
\begin{itemize}
    \item[(i)] $\bmod_\K(\emptyset) = 0.$ 
    \item[(ii)] If $\Gamma_1 \subset \Gamma_2$, then $\bmod_\K(\Gamma_1) \leq \bmod_\K(\Gamma_2).$ \hfill{(Monotonicity)} \index{monotonicity of transboundary modulus}
    \item[(iii)] If $\Gamma_1 < \Gamma_2$, then $\bmod_\K(\Gamma_1) \geq \bmod_\K(\Gamma_2).$ \hfill{(Overflowing)} \index{overflowing property of transboundary modulus}
    \item[(iv)] $\bmod_\K(\cup_{n \in \N} \Gamma_n) \leq \sum_{n \in \N} \bmod_\K(\Gamma_n)$. \hfill{(Subadditivity)} \index{subadditivity of transboundary modulus}
    \item[(v)] Suppose there are pairwise disjoint Borel sets, $B_n \subset X$, such that for all $i \in I$ there is at most one $n \in \N$ with $K_i \cap B_n \neq \emptyset$. If, for all $n \in \N$, $\gamma_n \subset B_n$ for all $\gamma_n \in \Gamma_n$, then $$\bmod_\K(\cup_{n \in \N} \Gamma_n) = \sum_{n \in \N} \bmod_\K(\Gamma_n).$$ 
\end{itemize}
\end{proposition}

We remark that one can make corresponding statements about families of curves in $X_\K$ by pulling them back to $X$. For example, if $\G_1,\G_2$ are collections of curves in $X_\K$ with $\G_1 \subset \G_2$, then one can find curve families $\G_1',\G_2'$ in $X$ with $\G_1' \subset \G_2'$ and $\pi_\K(\G_1')=\G_1$ and $\pi_\K(\G_2')=\G_2$, so that $\bmod_\K(\G_1) \leq \bmod_\K(\G_2)$.

A key property of  transboundary modulus is that, just like classical modulus, it is essentially preserved under quasiconformal mappings. Since our setting is more general than usual we provide the details for the reader's convenience. Similar lemmas on the quasiconformal quasi-invariance of transboundary modulus were also given in \cite{Bonk} and \cite{Hakobyan:Li}.

\begin{lemma}
\label{qc quasipreserves tm}
Let $(X,d_X,\mu_X),(Y,d_Y,\mu_Y)$ be a metric measure spaces with $\mu_X,\mu_Y$ locally finite Borel regular measures. Pick any countable collection of pairwise disjoint continua, $\K = \{K_i \}_{i \in I}$ in $X$ and $\mathcal{J}=\{ J_1 \}_{i \in I}$ in $Y$. Let $k_i = \pi_\K(K_i)$, $K=\cup_{i\in I} K_i$ and $j_i=\pi_\mathcal{J}(J_i)$, $J=\cup_{i \in I} J_i$. Suppose $X \setminus K$ and $Y \setminus J$ are homeomorphic to domains in $\RS$. Suppose there is a homeomorphism $f:X_\K \rightarrow Y_\mathcal{J}$ with $f(k_i)=j_i$ for all $i \in I$ such that $\pi_\mathcal{J}^{-1} \circ f \circ \pi_\K:X \setminus K \rightarrow Y \setminus J$ is geometrically $H$-quasiconformal for some $H \geq 1$.  Then 
\begin{align}\label{ineq:geom-qc-tr}
    H^{-1} \bmod_\K(\G) \leq \bmod_{\mathcal{J}}(f(\G)) \leq H \bmod_\K(\G)
\end{align}
for any family of curves  $\G$  in $X_\K$.
\end{lemma}
\begin{proof}
Throughout this proof, we will mildly abuse notation and write $f$ instead of $(\pi_\mathcal{J}^{-1} \circ f \circ \pi_\K)|_{X \setminus K}$. By Theorem \ref{Williams} we have that $f$ is analytically quasiconformal on $X \setminus K$. Williams \cite{Williams} showed that if $\g$ is any absolutely continuous curve in $X \setminus K$, $f$ is absolutely continuous on $\g$, and $\rho$ is a non-negative Borel function on $Y \setminus J$ with $A(\rho) < \infty$, then we have that 
\begin{align}\label{ineq:NL}
    \int_\g (\rho \circ f) g_f \ ds \geq \int_{f(\g)} \rho \ ds.
\end{align}

Every compact rectifiable curve has an absolutely continuous arc-length parameterization, see e.g. Proposition 5.1.8 in \cite{HKST}. Thus every locally rectifiable $\g$ in $X \setminus K$ satisfies (\ref{ineq:NL}) if $f$ is absolutely continuous on every compact subcurve of $\g$ and $A(\rho)<\infty$. Let $\G_f$ be the collection of curves in $X \setminus K$ on which $f$ is not absolutely continuous. Then $\bmod(\G_f)=0$ provided $f \in N^{1,2}_{\text{loc}}(X \setminus K:Y \setminus J)$, see e.g. \cite{Nages:Thesis}.     Now, let $\G_f^\K$ be the family of curves in $X_\K$ which contain a subcurve in $\pi_\K(\G_f)$. By the overflowing property we have $\bmod_\K(\G_f^\K)=0$.
Therefore 
$$\bmod_\K(\G \setminus \G_f^\K)\leq\bmod_\K(\G) \leq \bmod_\K(\G \setminus \G_f^\K) + \bmod_\K(\G_f^\K) = \bmod_\K(\G \setminus \G_f^\K),$$ and we have $\bmod_\K(\G \setminus \G_f^\K) = \bmod_\K(\G)$.

Now, pick any $\varrho=(\rho;\{ \rho_i \}_{i \in I}) \wedge_{\mathcal{J}} f(\G)$. Suppose that $\bmod_{\mathcal{J}}(f(\G)) < \infty$. Then we can assume that $A(\rho)< \infty$. Let $\rho'=(\rho \circ f) g_f$ and $\varrho'=(\rho';\{ \rho_i \}_{i \in I})$. We would like to sow that $\varrho'$ is admissible for $\G$ relative $\mathcal{K}$.

Pick any locally rectifiable $\g \in \G \setminus \G_f^\K$. Since $X \setminus K$ is homeomorphic to a domain in $\RS$, we can construct countably many curves $\g_m$ in $X \setminus K$ such that $\ell_{\rho'}^\K(\g) = \sum_m \ell_{\rho'}(\g_m)$, where each $\g_m$ is locally rectifiable (see the discussion preceding the definition of transboundary modulus). Since $\g \notin \G_f^\K$, we have that $f$ is absolutely continuous on any compact subcurve of $\g_m$ for every $m$ and therefore $\ell_{\rho'}(\g_m) \geq \ell_\rho(f(\g_m))$. From the definition of $\varrho'$ it follows that 
$\ell_{\varrho'}^\K(\g) \geq \ell_{\varrho}^{\mathcal{J}}(f(\g)) \geq 1,$ or that $\varrho' \wedge_\K \G \setminus \G_f^\K$. 

Finally, since $f$ is analytically quasiconformal we have
\begin{align*}
    \bmod_\K(\G) = \bmod_\K(\G \setminus \G_f^\K) &\leq \int_{X \setminus K} (\rho \circ f)^2 g_f^2 \ d\mu_X + \sum_{i \in I}\rho_i^2 \\
    &\leq H \bigg( \int_{X \setminus K} (\rho \circ f)^2 J_f \ d\mu_X + \sum_{i\in I}\rho_i^2 \bigg) \\
    &\leq H \bigg( \int_{Y \setminus J} \rho^2 \  d\mu_Y + \sum_{i \in I}\rho_i^2 \bigg). 
\end{align*}
By infimizing over all admissible $\varrho$, we obtain the right hand inequality in (\ref{ineq:geom-qc-tr}). Note that this inequality still holds if $\bmod_{\mathcal{J}}(f(\G)) = \infty$. {Applying the same argument to $f^{-1}$  we get the other inequality in (\ref{ineq:geom-qc-tr})}.
\end{proof}




Like modulus, transboundary modulus has nice properties for upper Ahlfors 2-regular spaces.

\begin{proposition}
\label{transboundary modulus of a point}
Let $(X,d,\mu)$ be a metric measure space with $\mu$ a locally finite Borel regular measure. Fix some $x_0 \in X$. Suppose there exists constants $C_0,R_0>0$ such that 
\begin{align}\label{ineq:upper-regularity}
    \mu(B(x_0,r)) \leq C_0 r^2
\end{align} for all $0<r<R_0$. Let $\K$ be any countable family of disjoint continua in $X$, none of which contain $x_0$, and suppose $X \setminus K$ is homeomorphic to a domain in $\RS$. Then if $\Gamma$ is a family of non-constant curves such that every $\g \in \G$ passes through $x_0$  then $$\bmod_\K(\G)=\bmod_{\K \cup  \{x_0\}}(\G)=0.$$
\end{proposition}
\begin{proof}
First, notice that there is some $r>0$ such that $B(x_0,r) \cap K = \emptyset$. Notice that every curve in $\G$ contains a non-constant subcurve in $B(x_0,r)$ that goes through $x_0$, call the family of these subcurves $\G'$. Then $\bmod_\K(\G) \leq \bmod_\K(\G')=0$, by Proposition \ref{transboundary properties}. To show that $\bmod_{\K \cup  \{x_0\}}(\G)=0$ let $\mathcal{J} = \K \cup \{ x_0 \}$.  Observe that if $\G_n = \{ \g \in \G \ | \ \g \nsubseteq B(x_0,1/n) \}$ then  $\G =\cup_n \G_n$ since no curves are constant. Thus, it suffices to show, for sufficiently large $n$, that $\bmod_\mathcal{J}(\G_n) = 0$. Take $N>n > 1/r$. Then every $\g \in \G_n$ contains a subcurve in $\G_{n,N}:=\G(\overline{B(x_0,1/N)},X \setminus B(x_0,1/n);X)$ and so $\bmod_\mathcal{J}(\G_n) \leq \bmod_\mathcal{J}(\G_{n,N}) = \bmod(\G_{n,N})$. By the well-known modulus bound for spaces with $\mu$ satisfying (\ref{ineq:upper-regularity}), see e.g. Proposition 5.3.9. in \cite{HKST}, we have  $\bmod(\G_{n,N})\leq {C}/{\log(N/n)}$, where $C$ depends only on $C_0$. Taking $N\rightarrow \infty$ shows then that $\bmod_\mathcal{J}(\G_n) = 0$.
\end{proof}

We remark that for modulus, one usually assumes the curves be non-constant because, for $\g(t)=x_0$, $\bmod(\{\g\})=\infty$. However, for transboundary modulus, if $\K$ contains $\{x_0\}$ we have $\bmod_\K(\{\g\})=1$ as the only admissible distributions give weight 1 to $\{x_0\}$. We wrap up our discussion of singletons in the following result.

\begin{corollary}
\label{singletons}
Let $(X,d,\mu)$ be a metric measure space with $\mu$ locally finite. Suppose $X$ is upper Ahlfors 2-regular. Let $\K,\mathcal{J}$ be countable collections of disjoint continua in $X$ with $\K \subset \mathcal{J}$. Suppose $X \setminus K$ and $X \setminus J$ are homeomorphic to domains in $\RS$. Suppose that $\mathcal{J} \setminus \mathcal{K}$ consists only of singletons which are all isolated in $J$. Then for any family of non-constant curves $\G$ in $X$, $$\bmod_\K(\G)=\bmod_\mathcal{J}(\G).$$
\end{corollary}
\begin{proof}
Let $\mathcal{J} \setminus \K = \{\{s_i\}\}_{i \in I}$. Let $\G_i$ be the family of non-constant curves in $X$ which go through $s_i$. Define $\G' =\G \setminus \cup_{i \in I} \G_i$. Since none of the curves in $\G'$ intersect anything in $\J \setminus \K$, we have $\bmod_\J(\G') = \bmod_\K(\G').$ Thus it suffices to show that $\bmod_\J(\G)=\bmod_\J(\G')$ and $\bmod_\K(\G)=\bmod_\K(\G')$. Since $\G' \subset \G$, monotonicity gives one direction on both equalities. On the other hand,  by subadditivity of tansboundary modulus and by Proposition \ref{transboundary modulus of a point} we have $\bmod_\K(\cup_{i \in I}\G_i) \leq \sum_{i \in I}\bmod_\K(\G_i)=0.$ Since $\G=\G'\cup (\cup_{i\in I} \G_i)$ it follows that
\begin{align*}
    \bmod_\K(\G) 
&\leq \bmod_\K(\G') + \bmod_\K(\cup_{i \in I}\G_i) = \bmod_\K(\G').
\end{align*} 
For $\J$, notice that because each $s_i$ is isolated, we can say that $X \setminus J_i$ is homeomorphic to a domain where $J_i = J \setminus \{s_i \}$. By Proposition \ref{transboundary modulus of a point}, we have $\bmod_\J(\G_i)=0$ and by subadditivity  $\bmod_\J(\cup_{i \in I}\G_i)=0$. Therefore, just like for $\mathcal{K}$, we have  $\bmod_\J(\G) \leq \bmod_\J(\G').$ As noted above, this completes the proof.
\end{proof}

\section{Transboundary Loewner Properties}\label{Section:tlp}

As usual we assume that $(X,d)$ is a separable metric space and $\mu$ is a locally finite Borel regular measure on $X$. 

We say $X$ is a \textit{Loewner space} if there is a decreasing function $\Psi:(0,\infty) \rightarrow (0,\infty)$ such that for all disjoint continua $E,F \subset X$ we have 
\begin{align}\label{loewner}
    \bmod(\G(E,F;X)) \geq \Psi(\Delta(E,F)).
\end{align}

The notion of a Loewner space was introduced by Heinonen and Koskela \cite{HK:Acta} and has since become one of the most important concepts in analysis on metric spaces. Next we define transboundary analog of the Loewner property. This has appeared implicitly in \cite{Bonk} and \cite{Merenkov} and more explicitly in \cite{Hakobyan:Li}. The definition given below differs slightly from these appearances.


Suppose that $X$ is homeomorphic to a domain in $\RS$, $\partial X$ is compact, and $\partial_0 X$ is countable. We say $X$ is \textit{transboundary Loewner} \index{transboundary Loewner} if there is a decreasing function $\Psi:(0,\infty) \rightarrow (0,\infty)$ such that 
\begin{align}\label{trans-loewner}
    \bmod_{\partial_0 X}(\G(E,F;\overline{X})) \geq \Psi(\Delta(E,F)),
\end{align}
whenever $E$ and $F$ are disjoint and non-degenerate continua in $X$.

From the definitions above it immediately follows that if $X$ is Loewner then it is also transboundary Loewner.



In \cite{Bonk} Bonk studied spherical domains for which neither (\ref{loewner}) nor (\ref{trans-loewner}) holds, however a similar condition still holds if one considers a different conformal invariant. In this paper we take Bonk's condition as a starting point of our study and give the following definition. We recall that $X$ is assumed to be a separable metric space equipped with a locally finite Borel regular measure $\mu$.




Suppose $X$ is homeomorphic to a domain in $\RS$, $\partial X$ is compact, and $\partial_0 X$ is countable. Given an integer $N\geq 0$, we say $X$ is \textit{N-transboundary Loewner}  if there is a decreasing function $\Psi:(0,\infty) \rightarrow (0,\infty)$ such that for all disjoint, non-degenerate continua $E$ and $F$ in $X$ we have
\begin{align}
\label{N-tlp}
\bmod_{\partial_0 X}(\G(E,F;\overline{X} \setminus K)) \geq \Psi(\Delta(E,F))
\end{align}
whenever $\mathcal{K} \subset \partial_0 X$ with $\#(\mathcal{K}) \leq N$.  

If $X$ is Loewner, then it will be $N$-transboundary Loewner for all $N$: $$\bmod_{\partial_0 X}(\G(E,F;\overline{X} \setminus K)) \geq \bmod_{\partial_0 X}(\G(E,F;X)) = \bmod(\G(E,F;X)).$$ Conversely, if $\partial_0 X$ is finite and $N \geq \#(\partial_0 X)$, then $X$ being $N$-transboundary Loewner gives $X$ is Loewner by selecting $\K=\partial_0 X$. Notice that being $0$-transboundary Loewner is identical to be transboundary Loewner, as $\overline{X} \setminus K = \overline{X}$. In fact, if $X$ is $N$-transboundary Loewner for any $N$, then $X$ is transboundary Loewner, as one can always take $\K = \emptyset$. More generally, for all $M,N \in \N$ with $M \geq N$, if $X$ is $M$-transboundary Loewner, then $X$ is $N$-transboundary Loewner. 

\begin{proposition}
\label{N-TLP is qs invariant}
Let $(X,d_X,\mu_X),(Y,d_Y,\mu_Y)$ be metric measure spaces with $\mu_X$ and $\mu_Y$ locally finite. Suppose $X$ and $Y$ are homeomorphic to domains in $\RS$, $\partial X, \partial Y$ are compact, and $\partial_0 X, \partial_0 Y$ are countable. Suppose $f:X \rightarrow Y$ is quasisymmetric and geometrically quasiconformal. If $X$ is $N$-transboundary Loewner, then $Y$ is $N$-transboundary Loewner.
\end{proposition}
\begin{proof}
Pick any non-degenerate, disjoint continua $A,B \subset Y$ and any $\K \subset \partial_0 Y$ with $\#(\K) \leq N$. Let $A' = f^{-1}(A)$ and $B' = f^{-1}(B)$, noticing that they are disjoint, non-degenerate continua in $X$. Since $f$ is quasisymmetric, $f$ extends to a homeomorphism of the completions. This means that $f$ gives rise to a homeomorphism $f:(\overline{X})_{\partial_0 X} \rightarrow (\overline{Y})_{\partial_0 Y}$. Let $\K' \subset \partial_0 X$ correspond to $\K$, and notice $\#(\K') \leq N$. Since $f$ is $N$-transboundary Loewner, $$\bmod_{\partial_0 X}(\G(A',B';\overline{X} \setminus K')) \geq \Psi'(\Delta(A',B'))$$ where $\Psi'$ is the function obtained from the transboundary Loewner property of $X$. Now apply Lemma \ref{qc quasipreserves tm} on $\overline{X}$ and $\overline{Y}$. Letting $f^{-1}$ be $\nu$-quasisymmetric and $f$ be $H$-quasiconformal, and using quasi-invariance of relative distance, we conclude
\begin{align*}
    \bmod_{\partial_0 Y}(\G(A,B;\overline{Y} \setminus K)) &\geq H^{-1} \bmod_{\partial_0 X}(\G(A',B';\overline{X} \setminus K')) \\
    &\geq H^{-1} \Psi'(\Delta(A',B')) \\
    &\geq H^{-1} \Psi'(\nu( 2 \Delta(A,B))).
\end{align*}
The function $\Psi(t) := H^{-1} \ \Psi'(\nu(2t))$ is decreasing since $\Psi'$ is decreasing and $\nu$ is increasing, so $Y$ is $N$-transboundary Loewner.
\end{proof}


\section{Examples of Transboundary Loewner Spaces}\label{Section:TLP-examples}

Because the transboundary Loewner property is a quasisymmetric invariant, it can be used to classify when spaces are quasisymmetrically equivalent. It is, therefore, of interest to generate some examples of model spaces which are transboundary Loewner. Our goal will be to derive conditions on $\K$, a collection of pairwise disjoint continua in some Loewner $\Omega$, which are sufficient to conclude $\Omega \setminus K$ is transboundary Loewner. Intuitively, these conditions will prohibit the continua in $\K$ from being too ``thin". In particular, we will show that circle domains are transboundary Loewner, which was first shown by Merenkov \cite{Merenkov} (it follows from Proposition 5.3). 

\begin{definition}[Schramm \cite{Schramm}]
Let $A \subset \R^2$ be Borel and $1 \geq \tau>0$. We say $A$ is ${\tau}$\textit{-fat} \index{fat sets} if, for all $x \in A$ and $r>0$ with $A \nsubseteq B(x,r)$, we have $$\Ha^2(A \cap B(x,r)) \geq \tau \Ha^2(B(x,r)).$$
\end{definition}

It's not hard to see that disks are $1/4$-fat. Rectangles are $(2 \pi c)^{-1}$-fat where $c$ is the ratio of the larger side length to the shorter one. Similarly, ellipses are fat with $\tau$ depending on the eccentricity.

For a general estimate on transboundary modulus, we will need the continua through which the curves go to be fat. This, however, won't be enough. We will also need the following property.

\begin{definition}
Let $A \subset \R^2$ and $\lambda \geq 1$. We say $A$ is ${\lambda}$\textit{{-quasiround}} \index{quasiround sets} if, there is some $x \in A$ and $r>0$ with $$B(x,r)\subset A \subset B(x,\lambda r).$$
\end{definition}

A set is $\lambda$-quasiround for some $\lambda$ if and only if it is bounded and has non-empty interior. What's important is that we keep track of how thick an annulus containing the boundary must be. Disks are $\lambda$-quasiround for any $\lambda > 1$. Like with fatness, rectangles and ellipses are quasiround with constants depending on their eccentricity. However, in general quasiround sets are not fat: take a Jordan domain with an outward pointing cusp; moreover, not all bounded fat sets are uniformly quasiround. That is, for sufficiently small $\tau$ and for all $\lambda \geq 1$, there is a bounded $\tau$-fat set which is not $\lambda$-quasiround. Take, for instance, $\mathbb{D} \setminus \alpha \mathbb{Z}\times \R$ for small $\alpha$. With more effort, one can come up with a family of Jordan domains which are $\tau$-fat but not uniformly quasiround.

\begin{proposition}[\cite{Bonk},\cite{Merenkov},\cite{Hakobyan:Li}]
\label{number of fat sets}
Let $E \subset \R^2$ be a continuum, $\lambda \geq 1$, and $\tau>0$. Let $\K=\{K_i\}_{i \in I}$ be a collection of pairwise disjoint, $\tau$-fat subsets of the plane satisfying $K_i \cap E \neq \emptyset$ and $$\lambda\diam(K_i) \geq \diam(E)$$ for all $i \in I$. Then $\#(\K) \leq (\lambda^2+6\lambda + 1)/\tau$.
\end{proposition}
\begin{proof}
Fix any point $e \in E$. Since $E$ is a compact, $\diam(E)<\infty$. Notice $E \subset B[e,\diam(E)]$, and thus $K_i \cap B[e,\diam(E)] \neq \emptyset$ for all $i \in I$. Define $$I_1 = \{i \in I \ | \ K_i \nsubseteq B(e,(1+\lambda^{-1})\diam(E)) \}.$$ Define $I_2 = I \setminus I_1$. For $i \in I_1$, we can say there is some $x_i \in \partial B(e,\diam(E)) \cap K_i$. Notice that $$B(x_i,\lambda^{-1}\diam(E)) \subset B(e,(1+\lambda^{-1})\diam(E));$$ indeed, for $y \in B(x_i,\lambda^{-1}\diam(E))$, we can say $$|y-e| \leq |y-x_i|+|x_i-e| \leq \lambda^{-1}\diam(E)+\diam(E).$$ Therefore, for $i \in I_1$, we can say $B(x_i,\lambda^{-1} \diam(E))$ does not contain $K_i$. We can use fatness to say $$\Ha^2(K_i \cap B(x_i,\lambda^{-1} \diam(E))) \geq \tau \Ha^2(B(x_i,\lambda^{-1} \diam(E))).$$ Now, for each $i \in I_1$, we can say $B(x_i,\lambda^{-1} \diam(E)) \subset A[e,(1-\lambda^{-1}) \diam(E),(1+\lambda^{-1}) \diam(E)]$. Since each $K_i$ is disjoint, we can say 
\begin{align*}
    \pi(((1+\lambda^{-1}) \diam(E))^2-&((1-\lambda^{-1}) \diam(E))^2)=\Ha^2(A[e,(1-\lambda^{-1}) \diam(E),(1+\lambda^{-1}) \diam(E)]) \\
    &\geq \Ha^2(\cup_{i \in I_1}B(x_i,\lambda^{-1} \diam(E))) \\
    &\geq \Ha^2(\cup_{i \in I_1}K_i \cap B(x_i,\lambda^{-1} \diam(E))) \\
    &= \sum_{i \in I_1}\Ha^2(K_i \cap B(x_i,\lambda^{-1} \diam(E))) \\
    &\geq \tau \sum_{i \in I_1}\Ha^2(B(x_i,\lambda^{-1} \diam(E))) \\
    &= \tau \sum_{i \in I_1}\pi(\lambda^{-1} \diam(E))^2 \\
    &=\pi \tau\lambda^{-2}\diam(E)^2 \#(I_1).
\end{align*}
Thus $$\#(I_1) \leq \tau^{-1}\lambda^2( (1+\lambda^{-1})^2-(1-\lambda^{-1})^2) = \tau^{-1}\lambda^2(4\lambda^{-1})=\frac{4\lambda}{\tau}.$$
To estimate $\#(I_2)$, notice that, by compactness of $K_i$ there exists $x_i,y_i \in K_i$ with $$d(x_i,y_i) = \diam(K_i) \geq \lambda^{-1} \diam(E).$$ Thus, $B(x_i,\lambda^{-1} \diam(E))$ does not contain $K_i$. Use fatness to say $$\Ha^2(B(x_i,\lambda^{-1} \diam(E)) \cap K_i) \geq \tau \Ha^2(B(x_i,\lambda^{-1} \diam(E))).$$ For $i \in I_2$, we have $K_i \subset B(e,(1+\lambda^{-1})\diam(E))$. Therefore, using disjointness of $K_i$,
\begin{align*}
    \pi((1+\lambda^{-1})\diam(E))^2 &= \Ha^2(B(e,(1+\lambda^{-1})\diam(E))) \\
    &\geq \Ha^2(\cup_{i \in I_2}K_i) \\
    &=\sum_{i \in I_2} \Ha^2(K_i) \\
    &\geq \sum_{i \in I_2} \Ha^2(K_i\cap B(x_i,\lambda^{-1} \diam(E))) \\
    &\geq \tau \sum_{i \in I_2} \Ha^2(B(x_i,\lambda^{-1} \diam(E))) \\
    &= \tau \sum_{i \in I_2} \pi (\lambda^{-1}\diam(E))^2 \\
    &= \pi \tau \lambda^{-2} \diam(E)^{2}\#(I_2).
\end{align*}
Hence, $\#(I_2) \leq \tau^{-1}\lambda^2(1+\lambda^{-1})^2=\tau^{-1}(\lambda+1)^2.$ Thus, we have $\#(\K)=\#(I)=\#(I_1)+\#(I_2) \leq \tau^{-1}(4\lambda + (\lambda +1)^2)=\tau^{-1}(\lambda^2+6\lambda +1).$ 
\end{proof}

The following result, first shown by Bojarski \cite{Bojarski}, is well known. However, it is usually formulated for finite sums, and we'll need it for infinite sums; we give an argument here which mimics a proof given by Merenkov \cite{Merenkov}.

\begin{lemma}[Bojarski's Lemma]
\label{Bojarski}
Let $\{B(x_i,r_i)\}_{i \in I}$ be a countable collection of pairwise disjoint open balls in $\R^2$ and $\{a_i\}_{i \in I}$ a countable collection of non-negative real numbers. Let $\lambda \geq 1$ be given. There exists a constant, $c_\lambda$, depending only on $\lambda$, such that $$\int_{\R^2} \bigg(\sum_{i \in I} a_i \mathbb{I}_{B(x_i,\lambda r_i)}\bigg)^2 dA \leq c_\lambda \int_{\R^2} \bigg(\sum_{i \in I} a_i \mathbb{I}_{B(x_i,r_i)}\bigg)^2 dA = c_\lambda \pi \sum_{i \in I} a_i^2 r_i^2.$$
\end{lemma}
\begin{proof}
Let $\phi \in L^2(\R^2)$. Define the uncentered maximal operator of $\phi$ by $$M(\phi)(x,y) = \sup_{(x,y) \in B(z,r)} \frac{1}{\pi r^2}\int_{B(z,r)} |\phi| \ dA.$$ Notice that, in particular, for $(x,y) \in B(x_i,r_i)$, we have $$M(\phi)(x,y) \geq \frac{1}{\pi \lambda^2 r_i^2} \int_{B(x_i, \lambda r_i)}|\phi| \ dA.$$ Hence, $$\int_{B(x_i,r_i)} M(\phi) \ dA \geq \frac{1}{\lambda^2}\int_{B(x_i, \lambda r_i)}|\phi| \ dA.$$ Recall that $M$ is a bounded operator (see Duoandikoetxea \cite{Duo} Theorem 2.5 and following remarks): there exists a constant, $C$, such that $$||M(\phi)||_2 \leq C ||\phi||_2$$ for all $\phi \in L^2(\R^2)$. 
\begin{align*}
    \int_{\R^2} \sum_{i \in I} a_i \mathbb{I}_{B(x_i,\lambda r_i)}|\phi| \ dA &= \sum_{i \in I} a_i \int_{\R^2}\mathbb{I}_{B(x_i,\lambda r_i)}|\phi| \ dA \hspace{90pt}\text{(Monotone Convergence Theorem)} \\
    &=\sum_{i \in I} a_i \int_{B(x_i,\lambda r_i)}|\phi| \ dA \\
    &\leq \sum_{i \in I} a_i \lambda^2 \int_{B(x_i, r_i)}M(\phi) \ dA \\
     &\leq \lambda^2 \int_{\R^2} \sum_{i\in I} a_i\mathbb{I}_{B(x_i,r_i)}  M(\phi) \ dA \hspace{120pt} \text{(Disjointness)}\\
     &\leq \lambda^2 ||M(\phi)||_2 \bigg|\bigg|\sum_{i \in I}a_i\mathbb{I}_{B(x_i,r_i)}\bigg|\bigg|_2  \hspace{110pt}\text{(Cauchy-Schwarz)} \\
     &\leq C \lambda^2 ||\phi||_2 \bigg|\bigg|\sum_{i \in I}a_i\mathbb{I}_{B(x_i,r_i)}\bigg|\bigg|_2 
\end{align*}
for all $\phi \in L^2(\R^2)$. With no loss of generality, suppose $I \subset \N$ and let $I_n = \{1,...,n\} \cap I$. Let $\phi_n = \sum_{i \in I_n} a_i \mathbb{I}_{B(x_i,\lambda r_i)}$. Notice $\phi_n \in L^2(\R^2)$ for all $n$ and $\phi_n \leq \phi_m$ for $n \leq m$. Suppose $||\phi_n||_2 > 0$ for sufficiently large $n$ (otherwise, the result is trivial). We use the above inequality to say
\begin{align*}
    ||\phi_n||^2_2 &\leq  \int_{\R^2} \sum_{i \in I} a_i \mathbb{I}_{B(x_i,\lambda r_i)}|\phi_n| \ dA \leq C \lambda^2 ||\phi_n||_2 \bigg|\bigg|\sum_{i \in I}a_i\mathbb{I}_{B(x_i,r_i)}\bigg|\bigg|_2 \\
    ||\phi_n||_2  &\leq C \lambda^2 \bigg|\bigg|\sum_{i \in I}a_i\mathbb{I}_{B(x_i,r_i)}\bigg|\bigg|_2 \\
   \lim_{n\rightarrow \infty} ||\phi_n||_2  &\leq C \lambda^2 \bigg|\bigg|\sum_{i \in I}a_i\mathbb{I}_{B(x_i,r_i)}\bigg|\bigg|_2 .
\end{align*}
Once again, we use the monotone convergence theorem to conclude $$\bigg|\bigg|\sum_{i \in I}a_i\mathbb{I}_{B(x_i,\lambda r_i)}\bigg|\bigg|_2 = \lim_{n\rightarrow \infty} ||\phi_n||_2 \leq C \lambda^2 \bigg|\bigg|\sum_{i \in I}a_i\mathbb{I}_{B(x_i,r_i)}\bigg|\bigg|_2. $$ Squaring both sides gives the desired conclusion.
\end{proof}

The following lemma is similar to a lemma of Bonk \cite{Bonk}, although he assumed uniform relative separation, and there is no such assumption here. It is also similar to a lemma of Merenkov \cite{Merenkov}, though he showed it only for circle domains.

\begin{lemma}[\cite{Hakobyan:Li}]
\label{transboundary modulus greater than modulus}
Let $\Omega \subset \R^2$ be Borel and $\G$ a collection of curves in $\Omega$. Fix $\tau >0$ and $\lambda \geq 1$. Let $\K=\{K_i\}_{i \in I}$ be a countable collection of pairwise disjoint, $\tau$-fat, $\lambda$-quasiround continua in $\Omega$. Suppose $\Omega \setminus K$ is a domain. Then there are constants $c_1$ and $c_2$ depending only on $\lambda$ and $\tau$ such that $$\bmod_\K(\G) \geq \min(c_1,c_2 \bmod(\G)).$$
\end{lemma}
\begin{proof}
Let $c = (1+12\lambda+4\lambda^2)/\tau$. Let $c_1:=1/(8c^2)$. Suppose $\bmod_\K(\G) \leq 1/(8c^2)$; as otherwise, we have nothing to show. For $\epsilon < 1/(8c^2)$, let $P=(\rho;\{\rho_i\}_{i \in I}) \wedge_\K \G$ be such that $$A_\K(P) \leq \bmod_\K(\G) + \epsilon \leq \frac{1}{4c^2}.$$ Since each $K_i$ is $\lambda$-quasiround, we can find $x_i \in K_i$ and $r_i>0$ such that $$B(x_i,r_i) \subset K_i \subset B(x_i,\lambda r_i).$$ For brevity, let $B_i:= B(x,2\lambda r_i)$. Define $g:\Omega \rightarrow [0,\infty)$ as $$g=2\bigg(\rho \ \mathbb{I}_{\Omega \setminus K} + \sum_{i \in I}\frac{\rho_i}{\lambda r_i} \mathbb{I}_{B_i \cap \Omega} \bigg).$$ We claim $g \wedge \G$. To see this, pick any $\g \in \G$ and define $$I_\g := \{ i \in I \ | \ \g \cap K_i \neq \emptyset, \ 2\lambda \diam(K_i) \geq \diam(\g)\}.$$  Now if $i \in I \setminus I_\g$, then we either have that $\g \cap K_i = \emptyset$ or $$\diam(\g)>2\lambda \diam(K_i) \geq 4\lambda r_i = \diam(B_i);$$ the latter implies that $\g$ is not contained in $B_i$. Therefore,
\begin{align*}
    \ell_g(\g) &= 2 \bigg(\int_{\g \setminus K} \rho \ ds + \sum_{i \in I}\frac{\rho_i}{\lambda r_i} \ell(\g \cap B_i) \bigg) \\
    &\geq 2 \bigg(\int_{\g \setminus K} \rho \ ds + \sum_{\g \cap K_i \neq \emptyset, i\in I\setminus I_\g}\frac{\rho_i}{\lambda r_i} \ell(\g \cap B_i) \bigg) \\ 
    &\geq 2\bigg(\int_{\g \setminus K} \rho \ ds + \sum_{\g \cap K_i \neq \emptyset, i\in I\setminus I_\g}\rho_i \bigg) \\
    &= 2\bigg( \ell_P^\K (\g) - \sum_{i \in I_\g} \rho_i \bigg)
\end{align*}
Observe that $\rho_i \leq \sqrt{A_\K(P)}\leq 1/(2c)$; also, we can use Proposition \ref{number of fat sets} to say $\#(I_\g) \leq c$. Hence $\sum_{i \in I_\g}\rho_i \leq \frac{\#(I_\g)}{2c} \leq \frac{1}{2}.$ This gives us admissibility: $$\ell_g(\g) \geq 2\bigg(\ell_P^\K(\g) - \sum_{i \in I_\g} \rho_i\bigg) \geq 2\bigg(\ell_P^\K(\g) - \frac{1}{2}\bigg) \geq 1.$$

We will now use $g$ to estimate the modulus. Since $B(x_i,r_i)$ are pairwise disjoint, we can use Bojarski's Lemma (Lemma \ref{Bojarski}) to say

\begin{align*}
    \bmod(\G) &\leq \int_\Omega g^2 \ dA = 4\int_\Omega \bigg(\rho \ \mathbb{I}_{\Omega \setminus K} + \sum_{i \in I}\frac{\rho_i}{\lambda r_i} \mathbb{I}_{B_i \cap \Omega}  \bigg)^2 dA \\
    &\leq 8\int_\Omega \bigg(\rho \ \mathbb{I}_{\Omega \setminus K}\bigg)^2 + \bigg(\sum_{i \in I}\frac{\rho_i}{\lambda r_i} \mathbb{I}_{B_i \cap \Omega}  \bigg)^2 dA \\
    &=8 \bigg( \int_{\Omega \setminus K}\rho^2 \ dA + \int_\Omega \bigg( \sum_{i \in I} \frac{\rho_i}{\lambda r_i} \mathbb{I}_{B_i \cap \Omega}\bigg)^2 \ dA \bigg) \\
    &\leq 8 \bigg( \int_{\Omega \setminus K}\rho^2 \ dA + \int_{\R^2} \bigg(\sum_{i \in I} \frac{\rho_i}{\lambda r_i} \mathbb{I}_{B_i}\bigg)^2 \ dA \bigg) \\
    &\leq 8 \bigg( \int_{\Omega \setminus K}\rho^2 \ dA + c_{2\lambda}\pi\sum_{i \in I} \frac{\rho_i^2}{\lambda^2 r_i^2} r_i^2 \bigg) \\
    &= 8 \bigg( \int_{\Omega \setminus K}\rho^2 \ dA + \frac{c_{2\lambda}\pi}{\lambda^2}\sum_{i \in I} \rho_i^2 \bigg) \\
    &\leq 8 \max(1,\frac{c_{2\lambda}\pi}{\lambda^2}) \bigg( \int_{\Omega \setminus K}\rho^2 \ dA + \sum_{i \in I} \rho_i^2 \bigg) \\
    &\leq 8 \max(1,\frac{c_{2\lambda}\pi}{\lambda^2}) (\bmod_\K(\G)+\epsilon)
\end{align*}
By letting $c_2:=(8 \max(1,\frac{c_{2\lambda}\pi}{\lambda^2}))^{-1}$ and taking $\epsilon \rightarrow 0$, we obtain the result.
\end{proof}

\begin{corollary}
\label{Loewner implies TLP}
Let $\Omega \subset \R^2$ be a Borel set with countably many compact boundary components and $\K=\{K_i\}_{i \in I}$ a countable collection of pairwise disjoint continua in $\Omega$ which are $\tau$-fat and $\lambda$-quasiround for some $\tau>0$, $\lambda \geq 1$. Let $K=\cup_{i\in I}K_i$, and suppose $\Omega \setminus K$ is a domain. If $\Omega$ is Loewner then $\Omega \setminus K$ is transboundary Loewner.
\end{corollary}
\begin{proof}
Let $\Omega$ be Loewner with decreasing function $\Psi$. Let $$\Psi'(t)=\min(c_1,c_2\Psi(t)),$$ where $c_1,c_2$ are the constants from Lemma \ref{transboundary modulus greater than modulus}. Then for any disjoint, non-degenerate continua $E,F \subset \Omega \setminus K$, by Lemma \ref{transboundary modulus greater than modulus} we have $$\bmod_\K(\G(E,F;\Omega)) \geq \min(c_1,c_2 \bmod(\G(E,F;\Omega))) \geq \min(c_1,c_2\Psi(\Delta(E,F)))=\Psi'(\Delta(E,F)).$$ Let $\J$ be the set of all connected components of $\partial \Omega$. To show $\Omega \setminus K$ is transboundary Loewner, we need to show $$\bmod_{\J \cup \K}(\G(E,F;\overline{\Omega \setminus K})) \geq \Psi'(\Delta(E,F)).$$ Let $K'=\partial(\overline{\Omega} \setminus K)$, the boundary of the interior of $K$. Observe that $\pi_\K(K') = \pi_\K(K)$, and thus $\pi_\K(\Omega)=\pi_\K((\Omega \setminus K) \cup K')$.
\begin{align*}
    \bmod_{\J \cup \K}(\G(E,F;\overline{\Omega \setminus K})) &\geq \bmod_{\J \cup \K}(\G(E,F;(\Omega \setminus K) \cup K')) \\
    &= \bmod_{\K}(\G(E,F;(\Omega \setminus K) \cup K')) \\
    &= \bmod_{\K}(\G(E,F;\Omega)) \geq \Psi'(\Delta(E,F)).
\end{align*}
\end{proof}

\begin{example}
\label{Circle domains are TLP}
Countably connected circle domains in $\R^2$ and $\RS$ are transboundary Loewner.
\end{example}
\begin{proof}
Let $X$ be a countably connected circle domain in $\R^2$. Let $\K$ be the collection of non-degenerate, bounded complementary components of $X$ and $\Omega:= X \cup K$. Then $\K$ is a collection of closed disks, which are uniformly fat, uniformly quasiround continua. If $\Omega$ is Loewner, then we can apply Corollary \ref{Loewner implies TLP} to say $X$ is transboundary Loewner. If none of the boundary components of $X$ are points, then $\Omega$ is either the plane or a disk, both of which are Loewner. If some of the boundary components are points, then $\Omega$ is the plane or the disk with countably many points removed, and thus $\Omega$ is Loewner as well.

Now let $X$ be a countably connected circle domain in $\R^2$. To see this, note that the sphere is Loewner. If $X$ is the sphere, then $X$ is transboundary Loewner. If $X$ is the sphere with only countably many points removed, then $X$ is Loewner, which implies that $X$ is transboundary Loewner. Suppose that $X$ has at least one non-trivial complementary component: $D$. Rotate the sphere so that the center of $D$ is the north pole; such a rotation is conformal and quasisymmetric. Use stereographic projection to map $X$ into the plane. Recall that stereographic projection is conformal and it sends circles not touching the north pole to circles. Thus the image will be a bounded circle domain in the plane. Notice $\RS \setminus D$ is Loewner, and that its image under stereographic projection is bounded and linearly locally connected. Therefore, by a theorem of Heinonen and Koskela, see \cite{HK:Acta}, stereographic projection is quasisymmetric as a map from $\RS \setminus D$, and thus its restriction to $X$ is quasisymmetric. By quasi-invariance of transboundary modulus, the circle domain in the sphere is transboundary Loewner since the image in the plane is.
\end{proof}

One can make similar statements for countably connected square domains. Let us establish a larger class of uniformly fat, uniformly quasiround shapes. A subset of $\RS$ or $\R^2$ is called an open $\eta$-quasidisk \index{quasidisk}, if it is the quasisymmetric image of $\mathbb{D}$ for some $\eta$-quasisymmetry. A closed quasidisk is a quasisymmetric image of $\overline{\mathbb{D}}$. An $\eta$-quasicircle \index{quasicircle} is an $\eta$-quasisymmetric image of $\partial \mathbb{D}$. We say a collection of quasicircles or quasidisks is uniform if they are all $\eta$-quasisymmetric images of a circle or disk for the same $\eta$. If one has a Jordan curve which is a quasicircle, its interior will be a quasidisk. The converse is true since quasisymmetries extend to the boundary. We will exclude a vast amount of theory of quasicircles; restricting our discussion to the following properties.

\begin{proposition}[Bonk\cite{Bonk} (Proposition 4.3)]
An $\eta$-quasidisk is $\lambda$-quasiround with $\lambda$ depending only on $\eta$.
\end{proposition}

\begin{proposition}[Schramm \cite{Schramm} (Corollary 2.3)]
An $\eta$-quasidisk is $\tau$-fat with $\tau$ depending only on $\eta$.
\end{proposition}

It becomes immediate then, that one can use Corollary \ref{Loewner implies TLP} when $\K$ is a collection of uniform quasidisks.

\section{Examples of $N$-Transboundary Loewner Spaces}\label{Section:N-tlp-examples}

Let us establish some examples (and non-examples) of $N$-transboundary Loewner spaces.

\begin{example}
There are domains which are transboundary Loewner and not 1-transboundary Loewner.
\end{example}
\begin{proof}
We construct a domain in $\R^2$ as follows. Define the following polar rectangle for $n \in \N$, $n \geq 1$: $$A_n:=\{(r,\theta) \ | \ 1 \leq r \leq 2, \ 0 \leq \theta \leq 2\pi-\frac{1}{n} \} + (6n,0).$$ These polar rectangles are pairwise disjoint and uniformly separated. Let $\Omega = \R^2 \setminus \cup_{n=1}^\infty A_n$, and notice $\Omega$ is a domain. Since the plane is Loewner, and $A_n$ are uniformly fat and quasiround, we can say, by Corollary \ref{Loewner implies TLP}, that $\Omega$ is transboundary Loewner. 

To see that it fails to be 1-transboundary Loewner, let $E_n=\partial B((6n,0),0.9)$ and let $F_n=\partial B((6n,0),2.1)$. Notice $$\Delta(E_n,F_n)=\frac{1.2}{1.8}=\frac{2}{3}.$$ Now $$\G(B[(6n,0),1],\R^2 \setminus B((6n,0),2);A[(6n,0),1,2] \setminus A_n) < \G(E_n,F_n;\R^2 \setminus A_n).$$ By the overflowing property and a well-known modulus computation for polar rectangles, we have $$\bmod_{\partial_0 \Omega}(\G(E_n,F_n;\R^2 \setminus A_n)) \leq \bmod(\G(B[(6n,0),1],\R^2 \setminus B((6n,0),2);A[(6n,0),1,2] \setminus A_n)) = \frac{1/n}{\log(2)}.$$ For any decreasing function, $\Psi$, pick $n>(\Psi(2/3)\log(2))^{-1}$, then $$\bmod_{\partial_0 \Omega}(\G(E_n,F_n;\R^2 \setminus A_n)) < \Psi(2/3) = \Psi(\Delta(E_n,F_n)).$$ Hence it cannot be 1-transboundary Loewner.
\end{proof}

In the previous example, the complementary components failed to be uniform quasidisks. If we require our complementary components to be uniform quasidisks, we establish a large class of $N$-transboundary Loewner spaces. We will need the following result of Bonk to show this.

\begin{proposition}[Bonk \cite{Bonk} (Proposition 7.5)]
\label{Complement of quasidisks is Loewner}
Let $\{D_i\}_{i=1}^n$ be a finite collection of pairwise disjoint, closed $\eta$-quasidisks in $\RS$. Suppose $\Delta(D_i,D_j) \geq \alpha$ for all $i \neq j$. Then $$\RS \setminus \cup_{i=1}^n D_i$$ is Loewner with $\Psi$ depending only on $n, \eta,$ and $\alpha$.
\end{proposition}

\begin{proposition}[Bonk \cite{Bonk}]
\label{N-TLP URS quasidisks}
 Let $\K$ be countable collection of uniformly relatively separated, closed, $\eta$-quasidisks in $\R^2$, and suppose $\R^2 \setminus K$ is a domain. Then $\R^2 \setminus K$ is $N$-transboundary Loewner for all $N \in \N$ with decreasing function $\Psi$ depending only on $\eta$, the relative separation, and $N$.
\end{proposition}
\begin{proof}
Pick any disjoint, non-degenerate continua $E,F \subset \R^2 \setminus K$ and any $\J \subset \K$ with $\#(\J)\leq N$. $\eta$-quasidisks will be uniformly fat and quasiround, so we can use Lemma \ref{transboundary modulus greater than modulus} to say $$\bmod_\K(\G(E,F;\R^2 \setminus J)) \geq \min(c_1,c_2 \bmod(\G(E,F;\R^2 \setminus J))),$$ where $c_1$ and $c_2$ depend only on $\eta$. We use Proposition \ref{Complement of quasidisks is Loewner} and following remarks to conclude that $\R^2 \setminus J$ is Loewner with $\Psi$ depending only on $\eta$, the relative separation, and $N$. Since $\eta$ and the relative separation are fixed, we have the same $\Psi$ for all $\J$ of size $N$. 
\end{proof}
The same proof applies to bounded domains in the plane whose boundary components are uniform quasicircles which are uniformly relatively separated and whose union is closed. It also applies to quasidisks in the sphere satisfying the same assumptions by observing that stereographic projection is quasisymmetric on bounded subsets of the plane. This can be generalized for $N=1$. Much like how we don't need uniform relative separation to conclude transboundary Loewner, we don't need it for 1-transboundary Loewner.

\begin{proposition}
\label{1TLP}
Let $\K$ be countable collection of closed, $\eta$-quasidisks in $\R^2$, and suppose $\R^2 \setminus K$ is a domain. Then $\R^2 \setminus K$ is $1$-transboundary Loewner with decreasing function $\Psi$ depending only on $\eta$.
\end{proposition}
\begin{proof}
Pick any disjoint, non-degenerate continua $E,F \subset \R^2 \setminus K$ and any $\eta$-quasidisk, $K_i \in \K$. Use Proposition \ref{Complement of quasidisks is Loewner} to say $\R^2 \setminus K_i$ is Loewner with $\Psi$ depending only on $\eta$. Then use the fact that uniform quasidisks are uniformly fat and uniformly quasiround to use Proposition \ref{transboundary modulus greater than modulus}: $$\bmod_\K(\G(E,F;\R^2 \setminus K_i)) \geq \min(c_1,c_2 \bmod(\G(E,F;\R^2 \setminus K_i))) \geq \min(c_1,c_2 \Psi(\Delta(E,F))),$$ where $c_1,c_2$ and $\Psi$ depend only on $\eta$. Thus the same bound will hold for any choice of $K_i$.
\end{proof}
Once again, this argument will also apply to bounded domains in the plane whose boundary components are uniform quasicircles, as well as quasidisks in the sphere. This proposition suggests that 1-transboundary Loewner does not imply 2-transboundary Loewner.

\begin{example}
There are domains which are 1-transboundary Loewner and not 2-transboundary Loewner.
\end{example}
\begin{proof}
For $n \in \N$, $n\geq 2$, let $S_n = [3n,3n+1] \times [0,1]$ and $T_n = [3n+1+\frac{1}{n},3n+2+\frac{1}{n}] \times [0,1]$. Let $\Omega = \R^2 \setminus (\cup_n S_n \cup T_n)$. Notice that $\Omega$ is a domain. Since $S_n$ and $T_n$ are all squares, they are uniform quasicircles. By Proposition \ref{1TLP}, this domain must be 1-transboundary Loewner.

To see that it fails to be 2-transboundary Loewner, let $E_n = \{ 3n+1 + \frac{1}{2n}\} \times [\frac{1}{8},\frac{3}{8}]$ and $F_n = \{ 3n+1 + \frac{1}{2n}\} \times [\frac{5}{8},\frac{7}{8}]$. Then $\Delta(E_n,F_n)=1$ for all $n$. However, every curve connecting $E_n$ and $F_n$ which avoids $S_n$ and $T_n$ contains a subcurve connecting the vertical sides of one of three rectangles: $[3n+1,3n+1+\frac{1}{n}] \times [0,\frac{1}{8}], [3n+1,3n+1+\frac{1}{n}] \times [\frac{3}{8},\frac{5}{8}], [3n+1,3n+1+\frac{1}{n}] \times [\frac{7}{8},1]$. Thus by the overflowing property and subadditivity, we can conclude $$\bmod_{\partial_0 \Omega}(\G(E_n,F_n;\R^2 \setminus(S_n \cup T_n))) \leq \frac{1/n}{1/8}+\frac{1/n}{1/4}+\frac{1/n}{1/8} = \frac{20}{n}.$$ Thus, for any decreasing function $\Psi$, we can find an $n$ so that $\frac{20}{n} < \Psi(1)$, and we conclude $$\bmod_{\partial_0 \Omega}(\G(E_n,F_n;\R^2 \setminus(S_n \cup T_n))) < \Psi(1) = \Psi(\Delta(E_n,F_n)).$$
\end{proof}

The previous examples suggest that uniform relative separation of the boundary components is important for a space to be 2-transboundary Loewner. Indeed, it is necessary for circle domains.

\begin{example}
\label{URS circle domains are 2-TLP}
For a countably connected circle domain, the following are equivalent.
\begin{itemize}
    \item[(1)] It is $N$-transboundary Loewner for all $N \in \N$.
    \item[(2)] It is 2-transboundary Loewner.
    \item[(3)] The bounding circles are uniformly relatively separated.
\end{itemize}
\end{example}
\begin{proof}
Note that $(1) \Rightarrow (2)$ follows from the definitions.

$(3) \Rightarrow (1):$ If the circles are uniformly relatively separated, and none of them are points, then we have that the domain is $N$-transboundary Loewner for all $N$. If some of the boundary components are points, then Proposition \ref{Complement of quasidisks is Loewner} still holds with countably many points removed. So the complement of $N$ complementary components is still uniformly Loewner, and Lemma \ref{transboundary modulus greater than modulus} can still be used. We deduce $N$-transboundary Loewner.

\begin{figure}
    \centering
    \includegraphics[width=5cm, height=5cm]{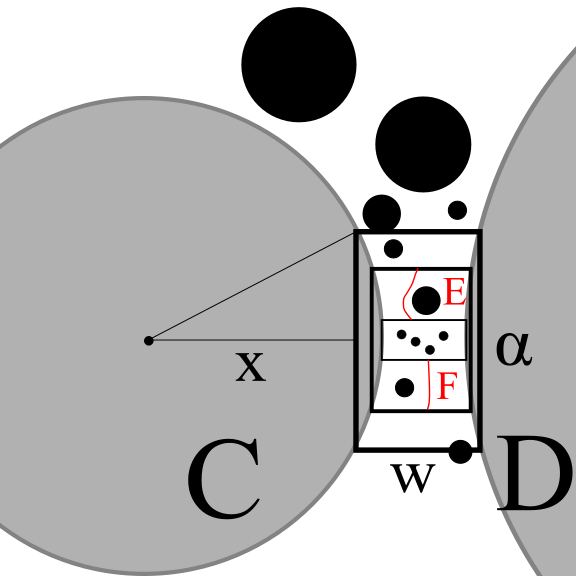}
    \caption{$C_n$ and $D_n$ are in gray and $E_n$ and $F_n$ are in red}
\end{figure}

$(2) \Rightarrow (3):$ We'll prove the contrapositive. Suppose that the circles fail to be uniformly relatively separated. It will suffice to show that the domain fails to be 2-transboundary Loewner. Let $\K$ be the collection of complementary components of the circle domain. Let $C_n,D_n \in \K$, $\diam(C_n) \leq \diam(D_n)$, be closed disks such that $\Delta(C_n,D_n) < \frac{1}{n}$. For simplicity, let us assume that the centers of $C_n$ and $D_n$ are in the x-axis, and that the center of $C_n$ is left of the center of $D_n$. We'll say a rectangle, $R$, is $\alpha$-good if its sides are parallel to the coordinate axes, the two left vertices are in the circle bounding $C_n$, the two right vertices are in the circle bounding $D_n$, the x-axis is the perpendicular bisector of the left and right sides, the centers of $C_n$ or $D_n$ are not in $R$, and the lengths of the left and right sides are $\alpha$. Let $$\alpha=\diam(C_n)\sqrt{1-(1-\Delta(C_n,D_n))^2}$$ Let $R^1_n$ be the $\alpha/4$-good rectangle, $R^2_n$ be the $3\alpha/4$-good rectangle, and $R^3_n$ be the $\alpha$-good rectangle. Let us compute the widths of these rectangles. To see this, note first that $R^3_n$ has the largest width, call it $w$. Consider the triangle created by the center of $C_n$, the top left corner of $R^3_n$, and some point $(x,0)$ intersecting the left side of $R^3_n$. Then the Pythagorean Theorem dictates $$x^2=(\diam(C_n)/2)^2-(\alpha/2)^2=(\diam(C_n)/2)^2(1-\Delta(C_n,D_n))^2.$$ Since $C_n$ has the smaller diameter, we can conclude
\begin{align*}
    w &\leq \diam(C_n)-2x+d(C_n,D_n) \\
    &=\diam(C_n)-\diam(C_n)(1-\Delta(C_n,D_n))+d(C_n,D_n) \\
    &=\diam(C_n)(1-(1-\Delta(C_n,D_n)) +\Delta(C_n,D_n)) \\
    &=\diam(C_n)(2\Delta(C_n,D_n))
\end{align*}
Briefly note that this implies $$w \leq \frac{2\alpha\Delta(C_n,D_n)}{\sqrt{1-(1-\Delta(C_n,D_n))^2}} = 2\alpha \sqrt{\frac{\Delta(C_n,D_n)}{2-\Delta(C_n,D_n)}} \leq 2\alpha.$$ Let $E_n$ be any continuum contained in the circle domain connecting the top edge of $R^2_n$ and the top edge of $R^1_n$. Let $F_n$ be any continuum connecting the bottom edge of $R^2_n$ and the bottom edge of $R^1_n$. Notice that $\min(\diam(E_n),\diam(F_n)) \geq \alpha/4$, and their distance is at most the diameter of $R^1_n$. In particular, $$d(E_n,F_n) \leq \sqrt{\alpha^2+w^2} \leq \sqrt{5}\alpha. $$ Thus $$\Delta(E_n,F_n) \leq 4\sqrt{5}.$$ 

Let $L^T_i$ and $L^B_i$ be the top and bottom sides of $R^i_n$ respectively. Let $\G_1:=\G(L^T_1,L^B_1;R^1_n \setminus (C_n \cup D_n))$, $\G_2:=\G(L^T_2,L^T_3;R^3_n \setminus (C_n \cup D_n))$, and $\G_3=\G(L^B_2,L^B_3;R^3_n \setminus (C_n \cup D_n))$. Observe that $$\G_1 \cup \G_2 \cup \G_3 < \G(E_n,F_n; \R^2 \setminus(C_n \cup D_n)).$$ Hence we can use overflowing and subadditivity to get $$\bmod_\K(\G(E_n,F_n;\R^2 \setminus (C_n \cup D_n))) \leq \bmod_\K(\G_1)+\bmod_\K(\G_2)+\bmod_\K(\G_3).$$

 Let $$\rho=\frac{4}{\alpha} \mathbb{I}_{R^1_n}.$$ If $K_i \cap R^1_n = \emptyset$, set $\rho_i=0$. Also, if $K_i = C_n$ or $K_i=D_n$, set $\rho_i=0$. Otherwise, let $$\rho_i=\frac{4h_i}{\alpha},$$ where $h_i = \diam(\pi_2(K_i \cap R^1_n))$. We claim $(\rho;\{\rho_i\}) \wedge_\K \G_1$, as for all $\g \in \G_1$, $$\int_{\g \setminus K} \rho \ ds + \sum_{K_i \cap \g \neq \emptyset}\rho_i= \frac{4}{\alpha}(\ell(\g \setminus K)+\sum_{K_i \cap \g \neq \emptyset} h_i) \geq \frac{4}{\alpha}(\Ha^1(\pi_2(\g \setminus K))+ \Ha^1(\pi_2(\g \cap K))) \geq 1.$$ Now use this to estimate the modulus:
\begin{align*}
    \bmod_\K(\G_1) &\leq \int_{R^1_n \setminus K}\rho^2 \ dA + \sum_{K_i \cap R^1_n \neq \emptyset}\rho_i^2 \\
    &\leq \frac{16}{\alpha^2}\bigg(A(R^1_n \setminus K)+\frac{4}{\pi}\sum_{K_i \cap R^1_n \neq \emptyset}\pi \bigg(\frac{h_i}{2}\bigg)^2 \bigg) \\
    &\leq \frac{16}{\alpha^2}\bigg(A(R^1_n \setminus K)+\frac{4}{\pi}\sum_{K_i \in \K}A(K_i \cap R^1_n) \bigg) \\
    &\leq \frac{64}{\pi \alpha^2}A(R_n^1) \leq \frac{16}{\pi \alpha} w = \frac{4}{\pi} \frac{w}{\alpha/4}.
    \end{align*}
    
    We remark that a similar argument will show 
    \begin{align*}
        \bmod_\K(\G_2) &\leq \frac{4}{\pi} \frac{w}{\alpha/8} = \frac{32}{\pi} \frac{w}{\alpha},\text{ and} \\
        \bmod_\K(\G_3) &\leq \frac{4}{\pi} \frac{w}{\alpha/8} = \frac{32}{\pi} \frac{w}{\alpha}.
    \end{align*}
    
    Thus the sum can be estimated as follows.
    
    \begin{align*}
    \bmod_\K(\G(E_n,F_n;\R^2 &\setminus (C_n \cup D_n))) \leq \frac{80}{\pi} \frac{w}{\alpha} \\ 
    &\leq \frac{80}{\pi}\frac{2\diam(C_n)\Delta(C_n,D_n)}{\diam(C_n) \sqrt{1-(1-\Delta(C_n,D_n))^2}} \\
    &= \frac{160}{\pi} \frac{\Delta(C_n,D_n)}{\sqrt{2\Delta(C_n,D_n)-\Delta(C_n,D_n)^2}} \\
    &=\frac{160}{\pi} \sqrt{\frac{\Delta(C_n,D_n)}{2-\Delta(C_n,D_n)}} < \frac{160}{\pi} \sqrt{\frac{1}{2n-1}}.
\end{align*}
For any decreasing function, $\Psi$, we can find an $n$ so that $$\bmod_\K(\G(E_n,F_n;\R^2 \setminus (C_n \cup D_n))) \leq \frac{160}{\pi} \sqrt{\frac{1}{2n-1}} < \Psi(4\sqrt{5}) \leq \Psi(\Delta(E_n,F_n)).$$
Hence, it cannot be 2-transboundary Loewner.
\end{proof}
The argument just presented assumed that $C_n$ and $D_n$ were disks. For a bounded circle domain, the failure to be uniformly separated may occur for $D_n$ being the unbounded complementary component, for all sufficiently large $n$. We remark that a similar argument will show that it fails to be 2-transboundary Loewner. Then, since circle domains in the sphere are quasisymmetric to bounded planar circle domains, we conclude the statement is true for spherical circle domains.

\section{Transboundary modulus in circle domains}\label{Section:circle-domains}

We are going to need a few facts about the geometry of circle domains as they are our model space for the quasisymmetric uniformization.

\begin{proposition}
\label{thicc annulus}
Let $z \in \R^2$ and $R>r>0$. Let $\mathcal{D}$ be a collection of pairwise disjoint, closed disks in $\R^2$ intersecting both $B[z,r]$ and $\R^2 \setminus B(z,R)$. If $$\frac{R}{r} \geq 14,$$ then $\#(\mathcal{D}) \leq 2$.
\end{proposition}
\begin{proof}
Let $D = B[z',r']$ be a closed disk intersecting $B[z,r]$ and $\R^2 \setminus B(z,R)$. Notice $r' \geq (R-r)/2$. If $z'=z$, then we have that $B(z,R) \subset D$, and any disk intersecting $B[x,r]$ is not disjoint with $D$. So suppose $z \neq z'$, and let $L$ be the ray starting at $z$ and going through $z'$. Every circle centered around $z$ intersects $L$ exactly once. Let $C$ be the circle of radius $(R+r)/2$ centered at $z$. Let $c$ be the intersection point of $L$ with $C$. We claim $c \in B[z',r'-(R-r)/2]$. To see this, we'll break it down into cases.

If $z' \in B[z,r]$, then $r' \geq R-|z'-z|$, and $$|c-z'| = (R-r)/2 + (r-|z-z'|) = (R+r)/2-|z'-z| \leq (R+r)/2+r'-R = r'-(R-r)/2.$$ If $z' \in A[z,r,R]$ then $r' \geq |z'-c|+(R-r)/2$, and $$|c-z'| \leq r'-(R-r)/2.$$ Finally, if $z' \in \R^2 \setminus B(z,R)$ then $r' \geq |z-z'|-r$ and $$|c-z'|=(R-r)/2+(|z-z'|-R)\leq (R-r)/2+(r+r'-R)=r'-(R-r)/2.$$

Pick any $w \in B[c,(R-r)/2]$. Then $$|w-z'| \leq |w-c| + |c-z'| \leq (R-r)/2 + r'-(R-r)/2 =r'.$$ Thus, $B[c,(R-r)/2] \subset D$. Let $C'$ be the circle of radius $(R+r)/4$ centered at $z$, and let $d$ be one of the two points satisfying $d \in C' \cap \partial B(c,(R-r)/2)$, and consider the triangle spanned by $d$, $c$, and $z$. Let $\theta$ be the angle corresponding to the vertex $z$. Law of Cosines tells us 
\begin{align*}
    \bigg(\frac{R-r}{2}\bigg)^2 &= \bigg(\frac{R+r}{2}\bigg)^2 + \bigg(\frac{R+r}{4}\bigg)^2 - 2 \frac{R+r}{2}\frac{R+r}{4}\cos(\theta) \\
    \cos(\theta) &=  \frac{5}{4}- \bigg( \frac{R-r}{R+r} \bigg)^2 =  \frac{5}{4}- \bigg( \frac{R/r-1}{R/r+1} \bigg)^2
\end{align*}

\begin{figure}
    \centering
    \includegraphics[width=8cm, height=8cm]{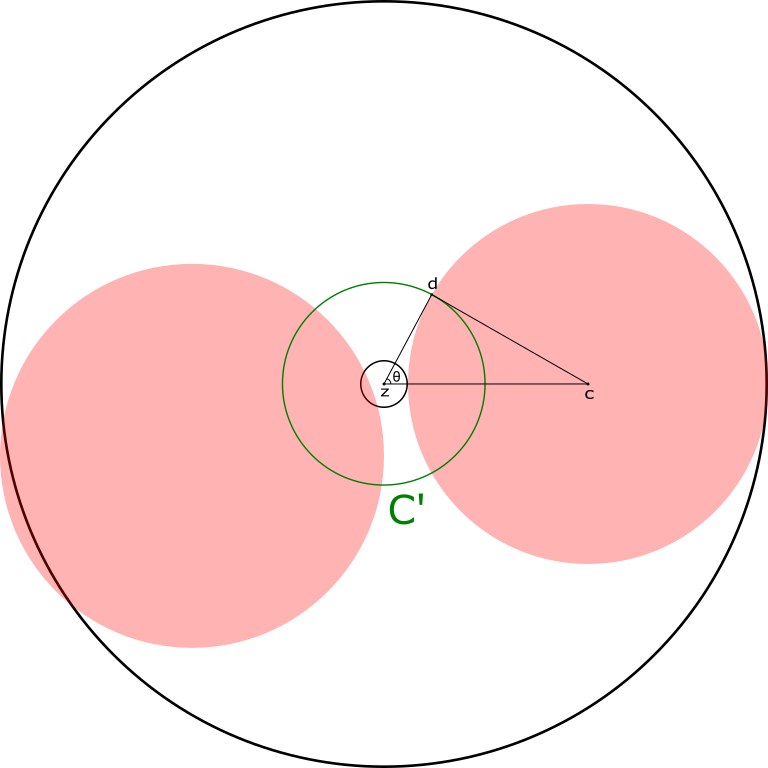}
    \caption{}
    \label{thicc}
\end{figure}

Notice that $\frac{x-1}{x+1}$ is an increasing function on $x \geq 1$. So $$\frac{R}{r} \geq 14 > \frac{2+\sqrt{3}}{2-\sqrt{3}} \Rightarrow \frac{R/r-1}{R/r+1} > \frac{\sqrt{3}}{2}.$$ Thus $$\frac{5}{4}- \bigg( \frac{R/r-1}{R/r+1} \bigg)^2 < \frac{1}{2}.$$ Since $\arccos(x)$ is a decreasing function, we have $$\theta > \arccos\bigg(\frac{1}{2}\bigg)=\frac{\pi}{3}.$$ Now we conclude $$\Ha^1(C' \cap D) \geq \Ha^1(C' \cap B[c,(R-r)/2]) = 2 \theta \frac{R+r}{4} > \frac{ \pi(R+r) }{6}.$$ This is true for all $D \in \mathcal{D}$. By disjointness, $$\Ha^1(C') \geq \Ha^1\bigg(\bigcup_{D \in \mathcal{D}}(C' \cap D)\bigg) = \sum_{D \in \mathcal{D}}\Ha^1(C' \cap D) > \frac{\pi}{6}(R+r)\#(\mathcal{D}).$$ Hence $$\#(\mathcal{D}) < \frac{6}{\pi (R+r)} \bigg(2\pi \frac{R+r}{4} \bigg) = 3.$$
\end{proof}
A similar statement will carry over to the sphere as stereographic projection is conformal and quasisymmetric on bounded subsets of the plane, and it sends circles in the plane to circles in the sphere. Quasisymmetries will distort the thickness of an annulus by $\eta$ and conformalities will preserve all the angles. 

\begin{lemma}[\cite{Bonk}]
\label{2-modulus goes to 0}
Let $\Omega \subset \R^2$ be a circle domain. If $\Omega$ is bounded, let $K_0$ be the bounding circle of its unbounded complementary component. Let $\K$ be the collection of the bounded complementary components of $\Omega$ as well as $K_0$, if $\Omega$ is bounded. Let $\Omega' = \Omega \cup K$. There is a decreasing function, $\Phi:(0,\infty) \rightarrow (0,\infty)$ with $\lim_{t \rightarrow \infty}\Phi(t)=0$ such that the following holds. If $E,F \subset \Omega$ are disjoint, non-degenerate continua with $\Delta(E,F) > \exp(\log(14)^{27})$, then there exists $\J \subset \K$ with $\#(\J) \leq 2$, such that $$\bmod_\K(\G(E,F;\Omega' \setminus J)) \leq \Phi(\Delta(E,F)).$$ 
\end{lemma}
\begin{proof}
We introduce the following notation for this proof. For an annulus, $A:=A(x,r,R)$ (or $A[x,r,R]$), and a set $C$, we define $$R^A_C = \sup_{y \in A \cap C} |x-y|$$ and $$r^A_C = \inf_{y \in A \cap C} |x-y|.$$ We will also use $$w_A(C) := \log \bigg( \frac{R^A_C}{r^A_C} \bigg),$$ which we will refer to as the width of the set. If $C$ is disjoint with $A$, we'll say the width is 0.

Suppose, without loss of generality, that $\min(\diam(E),\diam(F)) = \diam(E)$, and pick any $x \in E$. Let $r_0 = \diam(E)$ and $R_0 = d(E,F)$. Then $E \subset B[x,r_0]$ and $F \subset \R^2 \setminus B(x,R_0)$. Moreover, since $\Delta(E,F) > 1$, we have $r_0<R_0$. Let $A:=A[x,r_0,R_0]$ and notice that every curve connecting $E$ and $F$ contains a subcurve connecting the bounding circles of $A$.

We claim there is a subannulus, $A':=A[x,r',R'] \subset A$, with $R'/r' \geq 14$ and the following property: $$\# \{K_i \in \K \ | \ w_{A'}(K_i) > \log(R'/r')^{1/3} \} \leq 2.$$ If we have $$w_A(K_i) \leq \log(R_0/r_0)^{1/3}$$ for all $K_i$, then we're done as $A' = A$. Take $\mathcal{J} = \emptyset$. If not, let $D \in \K$ satisfy $$w_A(D) > \log(R_0/r_0)^{1/3},$$ and let $A_1 = A[x,r^A_D,R^A_D] \subset A$. Notice $w_{A_1}(D)=w_A(D)$ and $$\frac{R^A_D}{r^A_D} > \exp(\log(R_0/r_0)^{1/3}) > 14.$$ Now if $$w_{A_1}(K_i) \leq \log(R^A_D/r^A_D)^{1/3}$$ for all $K_i \neq D$, then we're done as $A' = A_1$. Take $\mathcal{J}=\{D\}$. If not, there exists $D' \in \K$, $D' \neq D$, with $$w_{A_1}(D') > \log(R^A_D/r^A_D)^{1/3},$$ and let $A_2 = A[x,r^{A_1}_{D'},R^{A_1}_{D'}] \subset A_1$. We claim $A' = A_2$ works. First notice $$\frac{R^{A_1}_{D'}}{r^{A_1}_{D'}} > \exp(\log(R^A_D/r^A_D)^{1/3}) > \exp(\log(R_0/r_0)^{1/9}) > 14.$$ Now if there is some $D'' \in \K$, $D'' \neq D',D'' \neq D$, with $$w_{A'}(D'') > \log(R'/r')^{1/3},$$ then $$\frac{R^{A'}_{D''}}{r^{A'}_{D''}} > \exp(\log(R'/r')^{1/3}) > \exp(\log(R_0/r_0)^{1/27}) > 14.$$ $D$ touches both bounding circles of $A_1$ and $D'$ touches both bounding circles of $A_2$. Thus we have $D,D',D''$ all touch both bounding circles of $A[x,r^{A'}_{D''},R^{A'}_{D''}] \subset A_2 \subset A_1$ which contradicts Proposition \ref{thicc annulus} if they're all disks. If one of them is $K_0$, then $(\R^2 \setminus \Omega') \cup K_0$ contains a closed disk touching both bounding circles of the annulus, and this disk is disjoint with the other 2, so it still contradicts Proposition \ref{thicc annulus}. Take $\mathcal{J}=\{D,D'\}$.

\hl{Is this part necessary?} Let $\{C_n\}$ be a countable collection of continua such that $C:=(\cup_n C_n) \cap A'$ is connected. Pick any $\epsilon > 0$. By virtue of being connected, if $R_{C_n}^{A'} < R_C^{A'}$ then there exists $n'\in\mathbb{N}$ with $R_{C_n}^{A'}<R_{C_{n'}}^{A'}$ and $r_{C_{n'}}^{A'}\leq R_{C_n}^{A'}(1+\epsilon)$. Define $G^\epsilon(n)$ to be the $n'$ value which maximizes $R_{C_{n'}}^{A'}$; such a value exists because its non-existence implies $C$ is disconnected. Similarly, if $r_{C_n}^{A'} > r_C^{A'}$ then there exists an $n'$ with $r_{C_n}^{A'}>r_{C_{n'}}^{A'}$ and $R_{C_{n'}}^{A'}(1+\epsilon)\geq r_{C_n}^{A'}$. Define $g^\epsilon(n)$ to be the $n'$ value which minimizes $r_{C_{n'}}^{A'}$; such a value exists because its non-existence implies $C$ is disconnected. Fix any $\delta > 0$ and take positive numbers $\{\epsilon_j\}_{j \in \mathbb{Z}}$ so that $\sum_{j \in \mathbb{Z}}(\log(1+\epsilon_j)) < \delta$. If there exists some $C_n$ not attaining the width of the union, one can define $n_j = G^{\epsilon_j}\circ ... \circ G^{\epsilon_1}(n)$ for integer $j > 0$, $n_j = g^{\epsilon_j}\circ ... \circ g^{\epsilon_{-1}}(n)$ for integer $j<0$ and $n_0=n$ with $\epsilon_0=0$, stopping at some finite point if we ever have $r_{C_{n_j}}^{A'}=r_C^{A'}$ or $R_{C_{n_j}}^{A'}=R_C^{A'}$. Since $C$ is connected, we must have $$w_{A'}(\cup_{j \in \mathbb{Z}} C_{n_j})=w_{A'}(C).$$ Thus we have
\begin{align*}
    \sum_n w_{A'}(C_n) &\geq \sum_{j \in \mathbb{Z}} w_{A'}(C_{n_j}) \\
    &=\lim_{j_0 \rightarrow \infty} w_{A'}(C_{n_0}) + \sum_{j=1}^{j_0} \log\bigg( \frac{R^{A'}_{C_{n_j}}}{r^{A'}_{C_{n_j}}} \bigg) + \sum_{j=-1}^{-j_0} \log\bigg( \frac{R^{A'}_{C_{n_j}}}{r^{A'}_{C_{n_j}}} \bigg) \\
    &= \lim_{j_0 \rightarrow \infty} w_{A'}(C_{n_0}) + \log\bigg(\prod_{j=1}^{j_0} \frac{R^{A'}_{C_{n_j}}}{r^{A'}_{C_{n_j}}} \bigg) +  \log\bigg( \prod_{j=-1}^{-j_0}\frac{R^{A'}_{C_{n_j}}}{r^{A'}_{C_{n_j}}} \bigg) \\
    &\geq \lim_{j_0 \rightarrow \infty} w_{A'}(C_{n_0}) + \log\bigg(\prod_{j=1}^{j_0} \frac{R^{A'}_{C_{n_j}}}{R^{A'}_{C_{n_{j-1}}}(1+\epsilon_j)} \bigg) +  \log\bigg( \prod_{j=-1}^{-j_0}\frac{r^{A'}_{C_{n_{j+1}}}}{(1+\epsilon_j)r^{A'}_{C_{n_j}}} \bigg) \\
    &= \lim_{j_0 \rightarrow \infty} w_{A'}(C_{n_0}) + \log\bigg( \frac{R^{A'}_{C_{n_{j_0}}}}{R^{A'}_{C_{n_0}}} \bigg) - \log\bigg(\prod_{j=1}^{j_0} (1+\epsilon_j) \bigg) +  \log\bigg( \frac{r^{A'}_{C_{n_0}}}{r^{A'}_{C_{n_{-j_0}}}} \bigg) -\log \bigg( \prod_{j=-1}^{-j_0}(1+\epsilon_j) \bigg) \\
    &= \lim_{j_0 \rightarrow \infty} \log \bigg( \frac{R^{A'}_{C_{n_{j_0}}}}{r^{A'}_{C_{n_{-j_0}}}} \bigg) - \sum_{|j|\leq j_0} \log(1+\epsilon_j) \\
    &\geq \lim_{j_0 \rightarrow \infty} \log \bigg( \frac{R^{A'}_{C_{n_{j_0}}}}{r^{A'}_{C_{n_{-j_0}}}} \bigg) - \delta \\
    &= w_{A'}(\cup_{j \in \mathbb{Z}} C_{n_j}) - \delta \\
    &= w_{A'}(C)-\delta.
\end{align*}
Take $\delta \rightarrow 0$ to say that the sum of the widths exceeds the width of the union. \hl{Seems like we could just assert that $w_A(C)$ is an outer measure.}

Now let $E' = B[x,r']$ and $F'=\R^2 \setminus B(x,R')$. Consider $\G:= \G(E',F'; \Omega' \setminus J)$. Let $\varrho=(\rho;\{\rho_i\}_{i \in I})$ be the transboundary mass distribution with $$\rho(z)=\frac{1}{\log(R'/r') |x-z|} \mathbb{I}_{A'},$$  and 

\begin{align*}
    \rho_i=\rho(K_i)=
    \begin{cases}
        0,& \mbox{ if } K_i\in\mathcal J \mbox{ or } K_i\cap A' =\emptyset,\\
        \displaystyle{\frac{w_{A'}(K_i)}{\log(R'/r')}},& \mbox{ otherwise.}
    \end{cases}
\end{align*}

Notice that, for any curve $\g'$ in $A' \setminus K$, we have $$\int_{\g'} \rho \ ds \geq \log(R'/r')^{-1} w_{A'}(\g').$$ 

Pick any $\g \in \G$. Let $\{\g_m\}$ be a collection of subcurves of $\g$ each corresponding to a connected component of $\g^{-1}(A' \setminus K)$ (see the discussion in Section \ref{Section:transboundary-modulus}). Notice that $\bigcup_m \g_m \cup \bigcup_{\g \cap K_i \neq \emptyset} K_i$ is connected. Therefore, the sum of the widths exceeds the width of the union, which is $\log(R'/r')$. Thus, 
\begin{align*}
    \int_{\g \setminus K} \rho \ ds + \sum_{\g \cap K_i \neq \emptyset} \rho_i 
    &=\log(R'/r')^{-1}\left(\sum_m \ell_{\rho}(\g_m)+ \sum_{\g \cap K_i \neq \emptyset} w_{A'}(K_i)\right)\\    
    &\geq \log(R'/r')^{-1}\left(\sum_m w_{A'}(\g_m)+ \sum_{\g \cap K_i \neq \emptyset} w_{A'}(K_i)\right)\\
    &\geq \log(R'/r')^{-1}\left(\log(R'/r')\right)=1.
\end{align*}
 Thus $\varrho \wedge_\K \G$. We will now use this to get an upper bound. 

Notice, for all $K_i \in \K$ intersecting $A'$ and $K_i \neq K_0$, we have that $K_i \cap A'$ contains a disk of radius $(R^{A'}_{K_i}-r^{A'}_{K_i})/2$. Moreover, $$\int_{A' \cap K_i}\rho^2 \ d\Ha^2 \geq \log(R'/r')^{-2}\frac{\Ha^2(A' \cap K_i)}{(R^{A'}_{K_i})^2} \geq \pi \log(R'/r')^{-2}\frac{(R^{A'}_{K_i}-r^{A'}_{K_i})^2}{4(R^{A'}_{K_i})^2}.$$ For $K_0$, if it intersects $A'$, let $K_0'=(\R^2 \setminus \Omega')\cup K_0$. Notice that $R^{A'}_{K_0}=R^{A'}_{K_0'}$ and $r^{A'}_{K_0}=r^{A'}_{K_0'}$. The above statement holds for $K_0'$.

Let $$\K_1 = \{ K_i \in \K \setminus \J \ | \ K_i \cap A' \neq \emptyset, w_{A'}(K_i) \leq \log(2) \}.$$ If $K_0 \in \K_1$, replace it with $K_0'$. For $K_i \in \K_1$, we have $R^{A'}_{K_i} \leq 2r^{A'}_{K_1}$, and $$\log\bigg( \frac{R^{A'}_{K_i}}{r^{A'}_{K_i}} \bigg) = \log\bigg(1+ \frac{R^{A'}_{K_i}-r^{A'}_{K_i}}{r^{A'}_{K_i}} \bigg) \leq \frac{R^{A'}_{K_i}-r^{A'}_{K_i}}{r^{A'}_{K_i}} \leq 2 \frac{R^{A'}_{K_i}-r^{A'}_{K_i}}{R^{A'}_{K_i}}.$$ 
Therefore we have
\begin{align*}
    \sum_{K_i \in \K_1} \rho_i^2 &\leq \log(R'/r')^{-2} \sum_{K_i \in \K_1}4 \bigg(\frac{R^{A'}_{K_i}-r^{A'}_{K_i}}{R^{A'}_{K_i}}\bigg)^2 \\
    &\leq \frac{16}{\pi} \sum_{K_i \in \K_1} \int_{A' \cap K_i} \rho^2 \ d\Ha^2 \leq \frac{16}{\pi} \int_{A'} \rho^2 \ d\Ha^2 = \frac{32}{\log(R'/r')}.
\end{align*}

Let $$\K_2 = \{ K_i \in \K \setminus \J \ | \ w_{A'}(K_i) > \log(2) \}.$$ If $K_0 \in \K_2$, replace it with $K_0'$. For all $K_i \in \K_2$, we have  
\begin{align*}
  \int_{A' \cap K_i}\rho^2 \ d\Ha^2 &\geq \pi \log(R'/r')^{-2}\frac{(R^{A'}_{K_i}-r^{A'}_{K_i})^2}{4(R^{A'}_{K_i})^2} \\
  &= \frac{\pi}{4}\log(R'/r')^{-2}\bigg(1-\frac{r^{A'}_{K_i}}{R^{A'}_{K_i}}\bigg)^{2} > \frac{\pi}{16}\log(R'/r')^{-2}.  
\end{align*}
 Thus we can say $$\frac{2\pi}{\log(R'/r')}=\int_{A'}\rho^2 \ d\Ha^2 \geq \sum_{K_i \in \K_2} \int_{A'\cap K_i}\rho^2 \ d\Ha^2 > \#(\K_2) \frac{\pi}{16}\log(R'/r')^{-2}.$$ Recall that for $K_i \notin \J$, we have $w_{A'}(K_i) \leq \log(R'/r')^{1/3}$. Use this to say 
 \begin{align*}
      \sum_{K_i \in \K_2} \rho_i^2 &\leq \#(\K_2)\log(R'/r')^{-2} \log(R'/r')^{2/3} \\
      &< \frac{16}{\pi} \log(R'/r')^{2/3} \frac{2 \pi}{\log(R'/r')} = \frac{32}{\log(R'/r')^{1/3}}.
 \end{align*}

Now we apply the modulus estimate. 
\begin{align*}
  \int_{\Omega}\rho^2 \ d\Ha^2 + \sum_{K_i \in \K}\rho_i^2 &\leq \int_{A'}\rho^2 \ d\Ha^2 + \sum_{K_i \in \K_1} \rho_i^2 + \sum_{K_i \in \K_2}\rho_i^2 \\
  &\leq \frac{2\pi}{\log(R'/r')}+\frac{32}{\log(R'/r')}+\frac{32}{\log(R'/r')^{1/3}}.  
\end{align*}
hence 
\begin{align*}
    \bmod_\K(\G(E',F';\Omega' \setminus J)) &\leq \frac{1}{\log(R'/r')^{1/3}}\bigg(\frac{2\pi + 32}{\log(14)^{2/3}}+32\bigg).
\end{align*}
Recall $\log(R'/r') > \log(R_0/r_0)^{1/9} = \log(\Delta(E,F))^{1/9}$. Now, since $A'$ separates $E$ and $F$, we can say every curve in $\G(E,F;\Omega' \setminus J)$ contains a subcurve in $\G$, and hence 
\begin{align*}
\bmod_\K(\G(E,F;\Omega' \setminus J)) &\leq \frac{1}{\log(\Delta(E,F))^{1/27}}\bigg(\frac{2\pi + 32}{\log(14)^{2/3}}+32\bigg). 
\end{align*} 
Denote by $\Phi(\Delta(E,F))$ the last term above. Observing that $\Phi$ decreases and $\lim_{t \rightarrow \infty} \Phi(t) = 0$, completes the proof.
\end{proof}

We remark that this result still holds in the sphere; since any spherical circle domain is quasisymmetric to a bounded planar circle domain. That is, for any $E$ and $F$ with sufficiently large relative distance, there are at most two complementary disks such that the transboundary modulus of curves avoiding those disks and connecting $E$ and $F$ is less than a function of their relative distance which goes to 0 as the relative distance goes to infinity. In fact, Bonk \cite{Bonk} proved this for finitely connected circle domains in the sphere; the proof just presented mimics his proof.

The final geometric fact about circle domains which we'll need concerns their linear local connectivity, see \cite{BonkKleinerMerenkov}.
\begin{proposition}
\label{circle domain LLC}
Circle domains are linearly locally connected with $\lambda=1$.
\end{proposition}

The following definition takes inspiration from the duality of modulus.

\begin{definition}[\cite{Rajala:unif}]
Let $(X,d,\mu)$ be a metric measure space with $\mu$ locally finite. Suppose $X$ is homeomorphic to a domain in $\RS$ and $\kappa \geq 1$. We say $X$ is $\boldsymbol{\kappa}$\textbf{-reciprocal} \index{reciprocal} if for all $Q \subset X$ homeomorphic to $[0,1]^2$ with $A_L,A_R,A_B,A_T \subset Q$ corresponding to $\{0\} \times [0,1], \{1\} \times [0,1],[0,1] \times \{0\},[0,1] \times \{1\}$ respectively, we have $$\frac{1}{\kappa} \leq \bmod(\G(A_L,A_R;Q)) \bmod(\G(A_B,A_T;Q)) \leq \kappa.$$ Moreover, we require that for all $x \in X$ and $R>0$ with $X \setminus B(x,R) \neq \emptyset$ we have $$\lim_{r \rightarrow 0} \bmod(\G(B[x,r],X \setminus B(x,R);X))=0.$$ We say $X$ is reciprocal if it is $\kappa$-reciprocal for some $\kappa$.
\end{definition}
The plane is reciprocal with $\kappa=1$. Reciprocality is a (geometrically) quasiconformal invariant. Its use below in quasiconformal uniformization of metric spaces homeomorphic to $\R^2$ gives a complete characterization.

\begin{theorem}[Rajala \cite{Rajala:unif}]
\label{QC reciprocal}
Let $(X,d)$ be a metric space with $\Ha^2$ locally finite. Suppose furthermore that $X$ is homeomorphic to $\R^2$. Then $X$ is (geometrically) quasiconformal to a domain in $\R^2$ if and only if $X$ is reciprocal.
\end{theorem}

\begin{theorem}[Rajala \cite{Rajala:unif}]
\label{AR reciprocal}
Let $(X,d)$ be a metric space with $\Ha^2$ locally finite. Suppose furthermore that $X$ is homeomorphic to $\R^2$. If $X$ is upper Ahlfors 2-regular, then $X$ is reciprocal.
\end{theorem}

This characterization was later generalized to include spaces not necessarily homeomorphic to the plane. We say $X$ is locally reciprocal \index{locally reciprocal} if for every $x \in X$, there is a neighborhood of $x$ which is reciprocal. If $X$ is homeomorphic to a domain in $\RS$, by Theorem \ref{QC reciprocal}, this is equivalent to stating every point has some neighborhood which is quasiconformal to a disk.

\begin{theorem}[Ikonen \cite{Ikonen}]
\label{Ikonen}
Let $(X,d)$ be a metric space with $\Ha^2$ locally finite. Suppose furthermore that $X$ is homeomorphic to a domain in $\RS$. If $X$ is locally reciprocal, then $X$ is (geometrically) $\frac{\pi}{2}$-quasiconformal to a Riemannian surface.
\end{theorem}

We are also going to need quasisymmetric maps to be geometrically quasiconformal.

\begin{lemma}[Ikonen \cite{Ikonen} (Lemma 6.5)]
\label{qs is qc}
Let $(X,d)$ be a metric space with $\Ha^2$ locally finite. Suppose $X$ is locally reciprocal and homeomorphic to a domain in $\RS$. If $f:\mathbb{D} \rightarrow V$ is $\eta$-quasisymmetric, for some $V \subset X$, then $f$ is (geometrically) $K$-quasiconformal with $K$ depending only on $\eta$.
\end{lemma}

\begin{lemma}
\label{technically *pushes up glasses*}
Let $(X,d)$ be a metric space with $\Ha^2$ locally finite. Suppose $X$ is locally reciprocal and $\Omega \subset \RS$ is a domain. If $f:\Omega \rightarrow X$ is $\eta$-quasisymmetric, then $f$ is (geometrically) $K$-quasiconformal with $K$ depending only on $\eta$.
\end{lemma}
\begin{proof}
Use the fact that $X$ is locally reciprocal and Theorem \ref{Ikonen} to say that $X$ is $\pi/2$-quasiconformal to a Riemannian surface. Koebe (\cite{Ahlfors-Sario}, Theorem 11C) proved that any Riemannian surface homeomorphic to a domain in $\RS$ is conformal to a domain in $\RS$. Thus we obtain a $\pi/2$-quasiconformal homeomorphism $g:X \rightarrow \Omega'$ for some domain $\Omega' \subset \RS$. Thus $g \circ f: \Omega \rightarrow \Omega'$ is a homeomorphism. By Lemma \ref{qs is qc}, $f$ is locally geometrically $K'$-quasiconformal where $K'$ depends only on $\eta$. So $g \circ f$ is locally geometrically $K'\pi/2$-quasiconformal. It is well known for homeomorphisms between domains in $\RS$ that local geometric quasiconformality is equivalent to global geometric quasiconformality (see Ahlfors\cite{Ahlfors:lectures}, for example). Hence $g \circ f$ is geometrically $K'\pi/2$-quasiconformal. Since $g^{-1}$ is geometrically $\pi/2$-quasiconformal, we get that $f$ is geometrically $K'\pi^2/4$-quasiconformal. Let $K:=K'\pi^2/4$.
\end{proof}

The final lemma we'll need is for quasiconformal maps to extend to the quotient spaces, so that transboundary modulus can be preserved. A sufficient condition comes from the property of bounded turning. 

\begin{lemma}[Merenkov-Wildrick \cite{Merenkov Wildrick}]
\label{quotient extension}
Let $(X,d)$ be a metric space with $\overline{X}$ compact. Suppose there is a homeomorphism $f:X \rightarrow \Omega$ where $\Omega \subset \RS$ is a domain. If $X$ is bounded turning, then $f$ extends to a homeomorphism $f:\overline{X}_{\partial_0 X} \rightarrow \overline{\Omega}_{\partial_0 \Omega}$.
\end{lemma}
Any homeomorphism will extend to the end compactifications of $X$ and $\Omega$. What Merenkov and Wildrick \cite{Merenkov Wildrick} showed was that if $X$ is bounded turning, then $\partial_0 X$ can be identified with the ends of $X$ homeomorphically. They also showed that for any domain $\Omega \subset \RS$, $\partial_0 \Omega$ is homeomorphic to the ends compactification of $\Omega$. Thus, $f$ extending to the end compactifications implies $f$ gives rise to a bijection, $f_0:\partial_0 X \rightarrow \partial_0 \Omega$, satisfying $(x_n)$ converges to a point in $K_i \in \partial_0 X$ if and only if $(f(x_n))$ converges to a point in $f_0(K_i) \in \partial_0 \Omega$.


\section{Proof of Main Result}\label{Section:proof}

In this section we prove Theorem \ref{forward direction}, which is the main result of this work.

The 2-transboundary Loewner property will enable upgrading a quasiconformal map to a quasisymmetric one. The following argument is of standard type for such results. In particular, it's similar to an argument by Bonk \cite{Bonk}; though it's adapted for a more general context.

\begin{lemma}
\label{main result}
Let $(X,d,\mu)$ be a metric measure space with $\mu$ locally finite. Suppose $X$ is metric doubling, $\overline{X}$ is compact, and $\partial_0 X$ is countable. Suppose there is a geometrically $Q$-quasiconformal map $f:X \rightarrow \Omega$ where $\Omega \subset \RS$ is a circle domain, which extends as a homeomorphism $f:\overline{X}_{\partial_0 X} \rightarrow \overline{\Omega}_{\partial_0 \Omega}$. If $X$ is 2-transboundary Loewner with decreasing function $\Psi$, then $f$ is $\eta$-quasisymmetric where $\eta$ depends only on $Q$, $\Psi$, $\diam(X)$, and the doubling constant.
\end{lemma}
\begin{proof}
Since $\overline{X}$ is compact, $\diam(X)<\infty$. There exists $a_0, a_\infty \in X$ such that $d(a_0,a_\infty) > 3\diam(X)/4$. Notice then that $B(a_0,\diam(X)/4) \cap B(a_\infty,\diam(X)/4) = \emptyset$. Since $X$ is homeomorphic to a domain in $\RS$, $X$ must be connected. Hence $B(a_0,\diam(X)/4) \cup B(a_\infty,\diam(X)/4) \neq X$. Let $a_1 \in X \setminus(B(a_0,\diam(X)/4) \cup B(a_\infty,\diam(X)/4))$. Let $\delta = \min(\diam(X)/4,1)$. Notice that $d(a_0,a_1),d(a_1,a_\infty),d(a_0,a_\infty) \geq \delta$. Compose $f$ with a M\"obius transformation so that $f(a_0)=0$, $f(a_1)=1$, and $f(a_\infty)=\infty$. Since M\"obius transformations are conformal, the composition is still $Q$-quasiconformal.

We first want to show that this map is weakly quasisymmetric. Suppose, by way of contradiction, that $f$ fails to be weakly quasisymmetric. Then for all $H>0$ there exists pairwise distinct $a,b,c \in X$ such that $d(a,b) \leq d(b,c)$ and $|a'-b'| > H|b'-c'|$ where $a'=f(a), b'=f(b), c'=f(c)$. $b',c' \in B(b',1.5|b'-c'|)$. Since $\Omega$ is linearly locally connected with $\lambda=1$ (Proposition \ref{circle domain LLC}), we have that there exists a continua $E' \subset \Omega$ connecting $b'$ and $c'$ such that $$\diam(E') \leq 3 |b'-c'|.$$ 

Since $\delta \leq 1$ we have $B(0,\delta/2),B(1,\delta/2),B(\infty,\delta/2)$ are pairwise disjoint: $b'$ must fail to be in two of them. Moreover, $B(a_0,\delta/2),B(a_1,\delta/2),B(a_\infty,\delta/2)$ are pairwise disjoint, so $a$ must fail to be in two of them. Thus there exists a $u' \in \{0,1,\infty\}$ such that $|u'-b'| \geq \delta/2$ and $d(a_{u'},a) \geq \delta/2$. 

Since $|a'-b'| \leq 2$, we have $$\frac{\delta}{4}|a'-b'| \leq \frac{1}{2}\delta \leq |u'-b'|.$$ So $u' \notin B(b',\delta/4|a'-b'|)$. Also, since $\delta \leq 1$, $a' \notin B(b',\delta/4|a'-b'|)$. We can use linear local connectivity (Proposition \ref{circle domain LLC}) to say there is an $F' \subset \Omega$ connecting $u'$ and $a'$ such that $$F' \cap B[b',\frac{\delta}{4}|a'-b'|] = \emptyset.$$ If $H \geq 24/ \delta$ then we have $$\frac{\delta}{4}|a'-b'| > 6|b'-c'|.$$ So we have $E' \subset B(b',3|b'-c'|) \subset B(b',\delta/4|a'-b'|)$, and hence $$d(E',F')\geq \frac{\delta}{4}|a'-b'|-3|b'-c'| >  \frac{\delta}{4}|a'-b'|-\frac{\delta}{8}|a'-b'| > \frac{H\delta}{8}|b'-c'|. $$ Since we have $\min(\diam(E'),\diam(F')) \leq 3|b'-c'|$, we can say $$\Delta(E',F') \geq \frac{\delta}{24}H.$$

Now let $E=f^{-1}(E')$ and $F=f^{-1}(F')$. Notice $E$ and $F$ are disjoint continua and $b,c \in E$, $a,a_{u'} \in F$. Observe that $$\diam(F) \geq d(a,a_{u'}) \geq \frac{\delta}{2} \geq \frac{\delta }{2 } \frac{\diam(E)}{\diam(X)}.$$ Since $\delta \leq \diam(X)$, we have $2\diam(X)/\delta \geq 1$. Thus $$d(E,F) \leq d(a,b) \leq d(b,c) \leq \diam(E) \leq \frac{2 \diam(X)}{\delta} \min(\diam(E),\diam(F)).$$ Thus $$\Delta(E,F) \leq \frac{2 \diam(X)}{\delta}.$$

Take $H$ large enough to apply Lemma \ref{2-modulus goes to 0} in $\Omega$. Letting $\K'$ and $\Phi$ be as in the lemma, we obtain $\J' \subset \K'$ with $\#(\J') \leq 2$, such that $$\bmod_{\partial_0 \Omega}(\G(E',F';\overline{\Omega} \setminus J')) = \bmod_{\K'}(\G(E',F';\RS \setminus J')) \leq \Phi(\Delta(E',F')) \leq \Phi(H(\delta/24)).$$ Since $f:\overline{X}_{\partial_0 X} \rightarrow \overline{\Omega}_{\partial_0 \Omega}$ is a homeomorphism, we can define $\J$ as the set of connected components of $J:=\pi_{\partial_0 X}^{-1}(f^{-1}(\pi_{\partial_0 \Omega}(J' \cap \overline{\Omega})))$. Notice $\#(\J) \leq 2$. Since $X$ is 2-transboundary Loewner, say with decreasing function $\Psi$, we can say $$\bmod_{\partial_0 X}(\G(E,F;\overline{X} \setminus J)) \geq \Psi(\Delta(E,F)) \geq \Psi(2 \diam(X)/\delta).$$ Now use Lemma \ref{qc quasipreserves tm} to say $$\bmod_{\partial_0 X}(\G(E,F;\overline{X} \setminus J)) \leq Q \bmod_{\partial_0 \Omega}(\G(E',F';\overline{\Omega} \setminus J')),$$ and thus, for all sufficiently large $H$, $$0<\Psi\bigg( \frac{2 \diam(X)}{\delta}\bigg) \leq Q \Phi \bigg(H \frac{\delta}{24} \bigg).$$ The above inequality must fail for sufficiently large $H$ because $\lim_{t \rightarrow \infty} \Phi(t)=0$; hence, we have a contradiction and $f$ must be weakly quasisymmetric. In fact, since the constant from Proposition \ref{2-modulus goes to 0} and the function $\Phi$ are entirely independent of $X$, $\Omega$, and $f$, we can say $H$ depends only on $Q$, $\Psi$, and $\diam(X)$.

Notice any subset of $\RS$ is doubling. Since $X$ is doubling, weak quasisymmetry implies quasisymmetry on $X$, see \cite[Theorem 10.19]{Heinonen}. Therefore, $f$ is $\eta$-quasisymmetric where $\eta$ depends on $H$ and the doubling constant.
\end{proof}

We will now use Ikonen \cite{Ikonen} to obtain a quasiconformal map through local reciprocality. We will also need to assume bounded turning to extend the map to the quotient spaces.

\begin{theorem}
\label{forward direction}
Suppose $(X,d)$ is a metric space with $\Ha^2$ locally finite. Let $\overline{X}$ be compact and $\partial_0 X$ countable. Suppose $X$ is (metric) doubling, bounded turning, homeomorphic to a domain in $\RS$, and locally reciprocal. If $X$ is 2-transboundary Loewner with decreasing function $\Psi$, then $X$ is $\eta$-quasisymmetric to a circle domain in $\RS$, where $\eta$ depends only on $\Psi$, $\diam(X)$, and the doubling constant.
\end{theorem}
\begin{proof}

Use the fact that $X$ is locally reciprocal and Theorem \ref{Ikonen} to say that $X$ is $\pi/2$-quasiconformal to a Riemannian surface. Koebe (\cite{Ahlfors-Sario}, Theorem 11C) proved that any Riemannian surface homeomorphic to a domain in $\RS$ is conformal to a domain in $\RS$. Since $X$ is bounded turning and $\overline{X}$ is compact, by Lemma \ref{quotient extension}, this domain must have countably many boundary components. By the well-known Theorem of He and Schramm, see \cite{HS} this domain is conformal to a circle domain. Thus we obtain a (geometrically) $\pi/2$-quasiconformal map $f:X \rightarrow \Omega$ where $\Omega \subset \RS$ is a countably connected circle domain. Moreover, by use of Lemma \ref{quotient extension}, we can say $f$ extends to be a homeomorphism $f:\overline{X}_{\partial_0 X} \rightarrow \overline{\Omega}_{\partial_0 \Omega}$. By Lemma \ref{main result}, we have that $f$ is $\eta$-quasisymmetric with $\eta$ depending only on $\Psi$, $\diam(X)$, and the doubling constant.
\end{proof}

Uniformization results for metric surfaces with finitely many boundary components given by Rajala-Rasimus \cite{RR} and Merenkov-Wildrick \cite{Merenkov Wildrick} have $\eta$ depend on the number of boundary components; it's worth pointing out that this result has no such dependency. However, there is a hidden dependency on the relative distance between boundary components, which is revealed in the following necessary conditions for a metric space to be 2-transboundary Loewner.

\begin{corollary}
\label{2TLP implies URS}
Suppose $(X,d)$ is a metric space with $\Ha^2$ locally finite. Let $\overline{X}$ be compact and $\partial_0 X$ countable. Suppose $X$ is metric doubling, bounded turning, homeomorphic to a domain in $\RS$, and locally reciprocal. If $X$ is 2-transboundary Loewner, then $X$ is $N$-transboundary Loewner for all $N$ and $\partial_0 X$ consists of uniformly relatively separated uniform quasicircles or points.
\end{corollary}
\begin{proof}
By Theorem \ref{forward direction}, we have a quasisymmetric and geometrically quasiconformal map to a circle domain in $\RS$. This will imply that $\partial_0 X$ consists of uniform quasicircles or points. Moreover, by Proposition \ref{N-TLP is qs invariant}, we have that the circle domain is 2-transboundary Loewner. Use Example \ref{URS circle domains are 2-TLP} to say the circle domain is $N$-transboundary Loewner for all $N$ and that its bounding circles are uniformly relatively separated. Use Proposition \ref{N-TLP is qs invariant} and the fact that quasisymmetric mappings preserve uniform relative separation, see e.g., \cite[Lemma 2.3]{Hakobyan:Li} to deduce that $X$ is $N$-transboundary Loewner for all $N$ and $\partial_0 X$ is uniformly relatively separated.
\end{proof}

Thus, in this context, 2-transboundary Loewner is equivalent to being $N$-transboundary Loewner for all $N$. So if $\partial_0 X$ is finite, we have $X$ is 2-transboundary Loewner if and only if it is Loewner, although the function making $X$ Loewner will depend on the number of boundary components. This also reduces greatly in the case of domains in the sphere, so that we can say the following.

\begin{corollary}
Let $\Omega \subset \RS$ be a countably connected domain. Then $\Omega$ is 2-transboundary Loewner if and only if its complementary components are uniformly relatively separated uniform quasidisks or points.
\end{corollary}
\begin{proof}
Every domain in $\RS$ will have $\Ha^2$ locally finite, $\overline{\Omega}$ compact, be metric doubling, and locally reciprocal. Also, looking at the discussion following Lemma \ref{quotient extension}, we can always extend homeomorphisms between spherical domains to a homeomorphism between the quotient spaces, so we don't need bounded turning. If $\Omega$ is 2-transboundary Loewner, by Corollary \ref{2TLP implies URS}, we have that the complementary components are uniformly relatively separated uniform quasidisks or points. The converse is given by Proposition \ref{N-TLP URS quasidisks}.
\end{proof}

Merenkov and Wildrick \cite{Merenkov Wildrick} showed that the boundary components of a metric space being uniformly relatively separated uniform quasicircles is insufficient to conclude quasisymmetric equivalence to a circle domain; despite the fact, which Bonk \cite{Bonk} showed, that it is sufficient for domains in $\RS$. However, the counterexample they gave fails to be transboundary Loewner, and hence it fails to be 2-transboundary Loewner. Thus, the above result shows that the 2-transboundary Loewner property is, in some sense, an appropriate perspective for generalizing Bonk's sufficient condition to metric spaces.

\begin{corollary}
\label{iff}
Suppose $(X,d)$ is a metric space homeomorphic to a domain in $\RS$ with $\Ha^2$ locally finite and $\partial_0 X$ countable. Suppose $X$ is locally reciprocal. Then $X$ is quasisymmetric to a circle domain with uniformly relatively separated bounding circles if and only if
\begin{itemize}
    \item $\overline{X}$ is compact,
    \item $X$ is (metric) doubling,
    \item $X$ is bounded turning, and
    \item $X$ is 2-transboundary Loewner.
\end{itemize}
\end{corollary}
\begin{proof}
Use Theorem \ref{forward direction} and Corollary \ref{2TLP implies URS} to conclude that any $X$ satisfying the assumptions is quasisymmetric to a circle domain with uniformly relatively separated boundary. For the other direction, notice that uniformly relatively separated circle domains in $\RS$ have all of the properties listed (see Example \ref{URS circle domains are 2-TLP}). Using Lemma \ref{technically *pushes up glasses*} and the fact that all of the listed properties are invariant under quasisymmetric and geometrically quasiconformal maps completes the proof.
\end{proof}

The statement here simplifies if we additionally assume Ahlfors regularity.

\begin{corollary}
\label{nice statement}
Suppose $(X,d)$ is a metric space homeomorphic to a domain in $\RS$ with $\partial_0 X$ countable. Suppose $X$ is Ahlfors 2-regular. Then $X$ is quasisymmetric to a circle domain with uniformly relatively separated bounding circles if and only if
\begin{itemize}
    \item $\overline{X}$ is compact,
    \item $X$ is bounded turning, and
    \item $X$ is 2-transboundary Loewner.
\end{itemize}
\end{corollary}
\begin{proof}
Observe that Ahlfors regularity gives $\Ha^2$ locally finite. It also gives metric doubling, since it gives the existence of a doubling measure on $X$. Use Theorem \ref{AR reciprocal} combined with the fact that $X$ is homeomorphic to a domain to conclude that every point in $X$ has a neighborhood which is reciprocal. Then apply Corollary \ref{iff} to reach the conclusion.
\end{proof}




\end{document}